%% file: main.tex
\documentclass{amsart}
\usepackage{amsmath,amsthm,amssymb,mathtools,bbm}
\usepackage{blkarray}
\usepackage{mathrsfs} 
\usepackage{graphicx} 
\usepackage[normalem]{ulem}   
\input{packages-paper}

\title[Categories of Type Charge--Conserving--with--Glue]{On Higher Representation Theory via Categories \\ of type Charge-Conserving--with--Glue}
\author{Paul P. Martin$^1$}
\address{$^1$University of Leeds}
\author{Sarah Almateari$^{1}$}
\author{Eric C. Rowell$^{1,3}$}
\address{$^3$Texas A\&M University}

\input{macros-mdef}

\input{macros-paper}

\input{macros-paper2}

\newcommand{\lav}{\lambda}
\newcommand{\mw}{\mu}
\newcommand{\off}{\overline{\ff}}
\newcommand{\unk}[2]{\raisebox{-5pt}{${#1 \atop #2}$}}
\newcommand{\unkk}[2]{\raisebox{-5pt}{$\overset{\scalebox{1.02}{$#1$}}{\scriptscriptstyle{#2}}$}} 
\newcommand{\unkkk}[2]{\raisebox{-4pt}{\begin{array}{c} #1 \\ {}^{#2}\end{array}}}

\begin{document}

\maketitle

\begin{abstract}
    In this paper we introduce a strict monoidal subcategory of the category of matrices,  
    suitable to address a higher representation theoretic analogue of radicals (non-semisimplicity) 
    in ordinary representation theory. We show the extent to which this analogue has analogous 
    representation theoretic properties. To illustrate, we apply to two key problems in the study of 
    braid representations (strict monoidal functors from the braid category $\Bcat$ to 
    the matrix category):  
    the classification problem; and the problem of analysing the 
    ordinary braid group representations that 
    braid representations generate
    in towers. 
\end{abstract}

\tableofcontents 

\section{Introduction}

Recently there has been progress classifying representations of 
certain strict monoidal categories 
by controlling the representing (or target) category.  
In particular   this approach uses   
target categories with 
intrinsic rigidity, such as  Match subcategories of matrix categories \cite{
MartinRowell2021}, 
also known as charge-conserving or CC categories; 
and addresses source categories such as the braid category \cite{mac2013categories}.
The limitation of such approaches is that some classes of representations may not have representatives in the
given rigid setting. 
For example, Hietarinta's classification \cite{Hietarinta92} of braid representations  in rank-2 
(in the categorical formulation these are natural functors $F:\Bcat \rightarrow \MAT^2$ from the braid category $\Bcat$ \cite{
mac2013categories} 
to the subcategory $\MAT^2$ of the category of matrices) 
contains representations with 
no Match equivalent. 
(This is not obvious. But we will give examples as part of our analysis.)
If the target constraints are relaxed completely then the classification problem is typically impossibly hard in general rank, 
so the question is how to enlarge the target to include a suitable 
transversal of representations without reaching the impossibility horizon. 
One attempt in this direction is the additive charge-conserving (ACC) class of categories \cite{HietarintaMartinRowell}. 
However, 
ACC coincides with CC in rank 2; and while
ACC braid representations are classified in rank 3,  in rank 4 this is already not practical 
(see e.g. \cite{Almateari} for a consideration of the complexity of the required ansatz). 
In the present paper we introduce a different enlargement of the Match categories, specifically motivated by the representations in Hietarinta's classification \cite{Hietarinta92} that are not of CC type.  
Remarkably, this enlargement is sufficiently capacious to welcome all but 
two 
of the classes found in \cite{Hietarinta92} 
(up to monoidal equivalence), 
and yet shows some promise of feasibility in rank-$3$, as we shall demonstrate.

In \S\ref{ss:constructCCwg} we construct the CCwg categories.
In \S\ref{ss:glue-calculus} we give the core general representation theoretic properties.
In \S\ref{ss:apps} we illustrate several applications of the method, both in constructing and classifying representations and also in analysing representation structure. 
In \S\ref{ss:repsN3} we apply this machinery to extend the CC classification of braid representations 
in rank 3 to CCwg. 
In \S\ref{ss:repsN2} we apply the machinery to determine representation structures in rank-2; and some in rank-3. 

\medskip 

Higher representation theory has been approached from several different directions. 
In ordinary representation theory many aspects of the problem align, resulting in a high degree of 
canonicalness in the way the problem can be formulated (see e.g. \cite{MartinRowell2021} for some discussion of this in comparison with the higher setting). 
In higher representation theory there is much less canonicalness, so that approaches necessarily start with 
some significant choice of assumptions, and essentially diverge from there. 
A good illustration of this is the difference between the approach of the school of Mazorchuk--Miemietz 
\cite{mazorchuk2015transitive,%
mazorchuk2017,Mackaay_2023}
 and that of Rumynin--Wendland \cite{RumyninWendland} (although there are others also worth comparing).
A third approach, from a different direction, is the `charge-conserving/target-narrowing' approach as in \cite{MartinRowell2021}. 
In a nutshell the difference between approaches is often triggered by the different intended applications. 
The Mazorchuk--Miemietz approach is what is sometimes called capital-T representation Theory - motivated by understanding the theory itself. Other approaches are motivated by the need to understand better the `source' category - the object that is being represented - capital-R Representation theory. 
This can be motivated in terms of understanding structures needed, for example, in topological quantum computation \cite{RowellWangBull}; and in the study of topological invariants \cite{Turaev88}. 

\vspace{.52cm}

\ignore{{\ssa{Conceptual overview.
Charge-conserving (CC) subcategories are defined by allowing mixing only  between words that have the same letters, possibly  in different order, while additive charge-conserving (ACC) subcategories allow  mixing between different compositions with the same total charge. The CC-with-glue construction introduced here lies between these two cases: it extends CC by permitting non-zero morphisms only between compositions that are comparable with respect to a fixed total order. This extension depends only on symbol counts, not on their positions within words.}\\
\ssa{[I added a brief conceptual overview in the introduction—would you mind letting me know if you think it’s helpful?]}}}

\subsection{Notation} 
\input{CCwithglue00}

\input{CCwithglue0}

\input{CCwithgluev2}

\section{Applications I: braid representations}  \label{ss:apps}

\newcommand{\AAA}{{\mathfrak{A}}}  

Given the CC-restriction property  (\ref{co:F_K}) we can 
in principle find all CCwg 
braid representations by formulating an ansatz consisting of a CC representation (from the complete classification) combined with a general glue matrix. Recall that for any rank $N$ and for any such ansatz 
$R \in \MAT(N^2,N^2)$ the {\em braid anomaly} is defined as
\beq   \label{eq:anomaly}
\AAA(R) =  
        R_1 R_2 R_1 -R_2 R_1 R_2 
 \eq 
where $R_1 = R\otimes 1_N$ and $R_2 =  1_N \otimes R$. 
The vanishing of $\AAA(R) $, together with invertibility of $R$, gives us a 
variety (possibly 0)  
of braid representations in glue parameter space. 

Here we will illustrate the method by investigating CCwg braid representations in rank 3, which rank is solved in the CC scheme but open in general. 

The CC classification scheme works up to various natural equivalences, principally local equivalence, which holds for all braid representations, but also up to X-symmetry - conjugation by a diagonal matrix \cite{MartinRowell2021}, which is special to the CC scheme. One thing to determine, therefore, is whether the extension to glue depends on the point taken in the X-orbit. We can already investigate this in rank 2, 
so we do this in \eqref{Xsym analysis}. 

For an extensive discussion of various notions of equivalence, see \cite{MRT25}.  We briefly mention a few here.
In the categorical parlance, local equivalence, i.e. when $R=(A\otimes A)S(A\otimes A)^{-1}$, corresponds to a monoidal natural equivalence.  We will also often refer to this as 'gauge equivalence', as it amounts to a local change of basis.  More generally, a natural equivalence is the same as what we will call $\infty$-equivalence meaning that the representations $\rho^R_n$ and $\rho^S_n$ corresponding to $R$ and $S$ are equivalent.  A finite form of this is $k$-equivalence, in which we demand that $\rho^R_n$ and $\rho^S_n$ are equivalent for all $n\leq k$.

\input{tex/N2}

\input{tex/N3}

\input{tex/N2-reps0}

\input{tex/N2-reps}

\ignore{{
\subsection{Representation theory in rank $N=3$}

...what (little) can perhaps be quickly done. 
Certainly we can do the next bits:

\subsubsection{Little-rank $n=2$}

...}}

\subsection{Representation theory of Unipotent braid representations}  \label{ss:unipotentRT}

Each R-matrix presents problems to analyse in ordinary representation theory. 
Restricting to unipotent representations gives a playground which is both relatively manageable and also 
quite illustrative of the features that appear. So we finish by investigating this area. 
Type-f has a unipotent case, as noted above, corresponding to Temperley--Lieb at $[2]=0$. 
This itself is well-understood \cite{Martin94b,MartinWoodcock_98}. We will briefly summarize here, since it establishes a base, via the HoG Theorem, from which to study the glue extensions in this and higher ranks. 
We have 
\[
R_0 = \mat 1 \\ &0&1 \\ &-1&2 \\ &&&1 \tam  
\]
As noted above, the entire with-glue extension is captured by the case 
\[
R = \mat 1&&&1 \\ &0&1 \\ &-1&2 \\ &&&1 \tam  
\]
In the following table we give all the relevant algebra dimensions in all little-ranks up to $n=6$. 

\begin{center}
    \begin{tabular}{c|cc|cc}
      n   & dim $\AAA^0$   & dim $\AAA^0$/rad & dim $\AAA$  & dim $\AAA$/rad  
      \\ \hline 
      3   & 5   &   5     & 6  &   5    \\
      4   & 14  &   5     & 20 & 5   \\ 
      5   &    42& 42   & 70 & 42  \\
      6   & 132  & 42  & 252 & 42
    \end{tabular}
\end{center} 

Observe that all these dimensions come from the classical structure, even directly or via the HoG theorem, 
except for the total algebra dimensions in the with-glue cases. Clearly the glue greatly enlarges the radical, 
but we do not yet have a closed form for general $n$ for that with-glue case.

In rank $N=3$ we have a unipotent fff representation; and a corresponding representative glue extension:
\[
R_{fff}^0 = 
\mat 1&   0&  0&    0& 0& 0&   0& 0& 0 \\ 
    0&   2&   0&    1& 0& 0&   0& 0& 0 \\ 
    0&   0&   2&    0& 0& 0&   1& 0& 0 \\ 
    0&  -1&   0&    0& 0& 0&   0& 0& 0 \\
    0&   0&   0&    0& 1& 0&   0& 0& 0 \\
    0&   0&   0&    0& 0& 2&   0& 1& 0 \\ 
    0&   0&  -1&    0& 0& 0&   0& 0& 0 \\  
    0&   0&   0&    0& 0& -1&  0& 0& 0 \\
    0&   0&   0&    0& 0& 0&   0& 0& 1
    \tam , 
    \hspace{.71cm}
    R_{fff}^{g} = 
\mat 1&   0&  0&    0& 0& 0&   0& 0& 0 \\ 
    0&   2&   0&    1& 0& 0&   0& 0& 0 \\ 
    0&   0&   2&    0& 0& 0&   1& 0& 0 \\ 
    0&  -1&   0&    0& 0& 0&   0& 0& 0 \\
    0&   0&   0&    0& 1& 0&   0& 0& 1 \\
    0&   0&   0&    0& 0& 2&   0& 1& 0 \\ 
    0&   0&  -1&    0& 0& 0&   0& 0& 0 \\  
    0&   0&   0&    0& 0& -1&  0& 0& 0 \\
    0&   0&   0&    0& 0& 0&   0& 0& 1
    \tam
\]

The dimensions up to little-rank 6 are:

 \begin{center}  
    \begin{tabular}{c|cc|cc}
      n   & dim $\AAA^0$   & dim $\AAA^0$/rad & dim $\AAA^g$  & dim $\AAA^g$/rad  
      \\ \hline 
      3   &   6 &    5    & 6   &    5   \\
      4   &   23 &   5     & 24   & 5    \\ 
      5   & 103  & 42   & 120 & 42  \\
      6   & 513  & 513-215 =298  & 695 & 695-397 =298
    \end{tabular}
\end{center}

Again most dimensions here come from the classical structure, but not the with-glue total dimensions, 
which are much larger than for the classical unipotent fff case. 
See Fig.\ref{fig:placeho}. 
The dimension computations are as follows. 
\\ 
In little-rank $n=3$ the Specht module labels are 3, 21, $1^3$ (we use integer partitions and Young diagrams interchangeably).
The 21 is simple and in a simple block (it is alone here in its affine reflection group orbit), 
of representation dimension 2, and hence contributing 4 to algebra dimension. 
The 3 and $1^3$ are 1d, but isomorphic here, and this module self-extends, so we add 1 to the 
dimension of the semisimple part, and 1+1 to the total dimension. 
\\
In little-rank $n=4$ the labels are 4, 31, $2^2$, $21^2$. 
By \cite{Martin94b} the $R^0_{fff}$ is a flat deformation of the classical `TLM' algebra \cite{Brzezinski_1995}, 
which is the group algebra of 
$S_4$ with the 1d $1^4$ module missing, so the dimension is 24-1=23. Here the semisimple quotient comes from 
the TL part. 
\\
In little-rank $n=5$ the labels are 5, 41, 32, $31^2$, $2^2 1$. That is partitions $21^3$ and $1^5$ are 
omitted. Recall these have rep dimensions 4 and 1, hence the dimension for $R^0_{fff}$ is 120-17=103. 
The semisimple dimension comes from the TL part \cite{MartinWoodcock_98,Soergel97a,Soergel97b} 
- and this is the same for the with-glue case by the HoG theorem. 
\\
In little-rank $n=6$ the orbit structure is indicated in purple in the figure.  
The computations are similar. The labels, with their representation dimensions beneath, are
\[
\begin{array}{ccccccc|cccc}
6  &  51  &  42  &  33  &  411  &  321  &  222  & 3111& 2211 & 21^4 & 1^6
\\
1  &   5  &   9  &   5  &   10  &   16  &    5  &   10&   9  &  5   & 1
\end{array}
\]
(we include the labels which are omitted at $N=3$, to the right of the bar)
so the dimension of the image 
of the CC R-matrix fff 
for generic (hence semisimple) Hecke is the sum of squares: 513; 
but by \cite{Martin94b} the unipotent case is a flat deformation, so has the same dimension, 
even though it does not have the same structure. 
The TL part of the structure is the line from $\lambda=(6)$ to $(33)$, with Specht module morphisms indicated in purple, so unpacking we have 
\[
(6) \rightarrow (51) \rightarrow (42) \rightarrow (33)
\]
(the chain of short reflections in the affine Wely group) giving Specht module Loewy structures (i.e. socle series) 
\[
1 \rightarrow {4\over 1} \rightarrow {5\over 4} \rightarrow {0\over 5}
\]
The remaining labels in the block are 411 and 222. These contribute nothing to the semisimple part 
(for example by considering idempotent localisation to little-rank $n=3$; or by an application of Kazhdan--Lusztig theory). Finally we have 321 with 
representation dimension 16, which is a simple singleton here. So overall the semisimple quotient dimension is 
$1^2 + 4^2 + 5^2 + 16^2 = 298$. 
By the HoG theorem this is also the semisimple quotient dimension in the with-glue case. 


\begin{figure}
    \centering
    \includegraphics[width=0.95\linewidth]{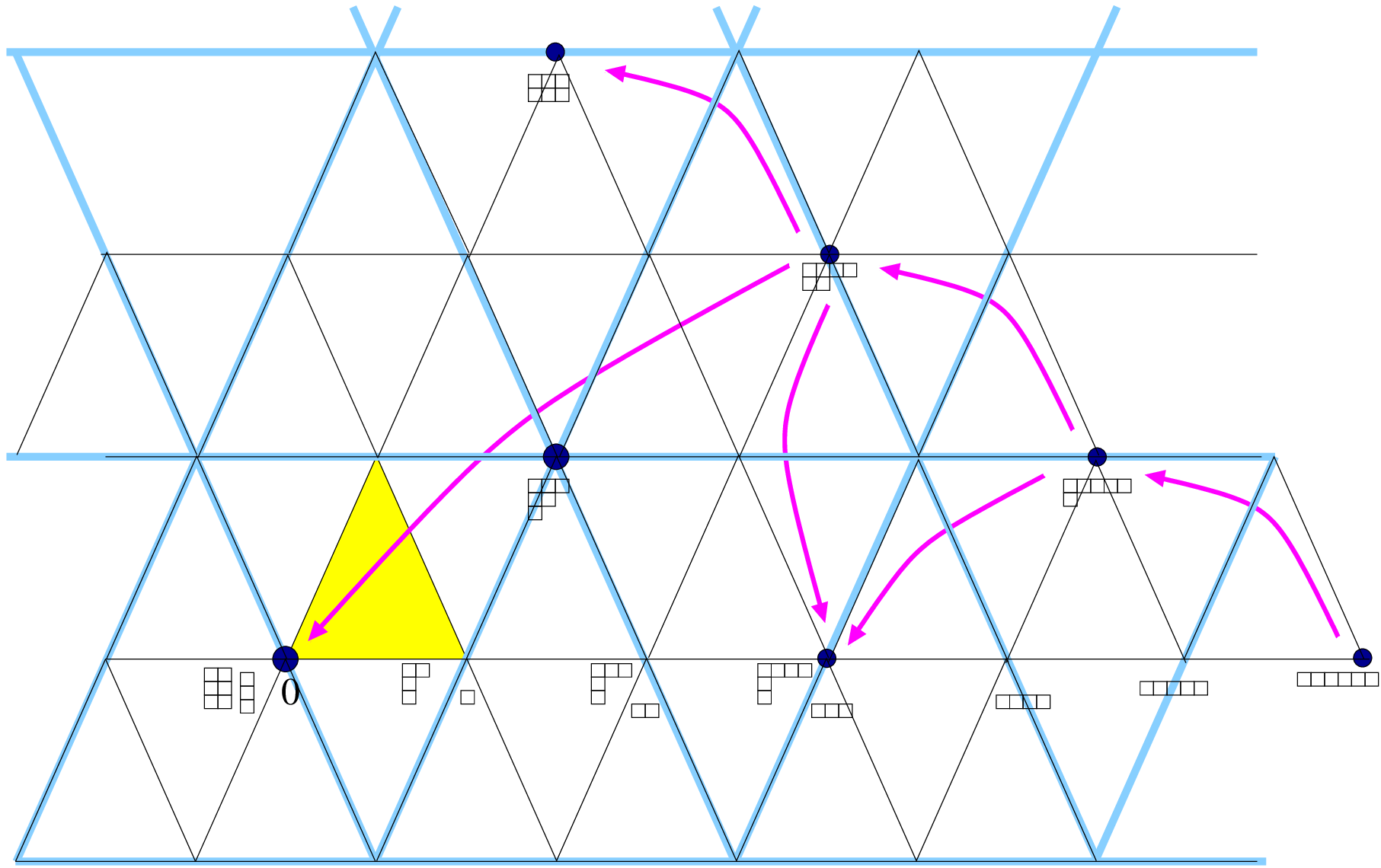}
    \caption{Alcove geometric framework for unipotent fff. 
    This is the $\Z^3$ nearest neighbour lattice projected along the (1,1,1) line; with affine reflection walls at separation 2 (in blue). 
    The origin and the start of the dominant region (in Lie terminology) is marked with a yellow triangle.}
    \label{fig:placeho}
\end{figure}


\medskip 

The pattern in the $R_{fff}^0$ case, and the semisimple part of the with-glue case, will by now be clear.  
And observe that this is enough to show the large contribution made to the radical by the glue extension. 
We leave an analysis of the new structure to a separate work 
(there will be a nice interplay between idempotent localisation and the R-matrix tower unification here, but it is far from trivial). 

\medskip 

\section{Discussion}
We have shown that there are braid representations in $\MATCHwg$ that are not equivalent, up to various interpretations, to those found in known classifications. At the moment this amounts to construction, rather than classification. It is not infeasible that a classification in rank $3$ can be carried out. Moreover, the evidence in rank $2$ suggests that the CCwg cases cover a wide swath of all such representations, up to $\infty$-equivalence.  The images of these representations are also intriguing, and deserve future attention.

\vspace{.53cm} 

\noindent {\bf Acknowledgements}.
PM thanks EPSRC for funding under project EP/W007509/1; 
and thanks the Royal Society for funding ER's Wolfson Fellowship at Leeds 
contributing to this work. PM also thanks 
Jarmo Hietarinta, 
Fiona Torzewska 
and Paula Martin for useful discussions. 
The work of ER was also partially supported by US NSF grant DMS-2205962.  Part of this work was carried out while ER was visiting the University of Leeds on a Royal Society Wolfson Fellowship, and he thanks them for their hospitality and support.

\bibliographystyle{alpha}
\bibliography{tex/references,tex/refs2,tex/more}

\end{document}

%% file: packages-paper.tex
\usepackage{amsthm}
\usepackage{amsfonts}
\usepackage{hyperref}
\usepackage{xcolor}
\usepackage{lmodern}
\usepackage{pifont}
\usepackage{breqn}
\usepackage{amssymb}
\usepackage{tikz}
\usepackage{tikz-cd}
\usepackage{pst-node}
\usetikzlibrary{decorations.pathmorphing}
\usepackage{tkz-euclide}
\usepackage{pst-plot}
\usetikzlibrary{positioning}
\usepackage{pgfplots}
\usepackage{graphicx} 
\usepackage{pifont}
\usepackage{amsmath}
\usepackage{pgfplots}
\pgfplotsset{compat = newest}
\usepackage{youngtab}
\usepackage{ytableau}
\usepackage{pictexwd,dcpic}
\usepackage[all]{xy}


%% file: macros-mdef.tex
\newcounter{minidef}[section]

\newcommand{\mdef}{\refstepcounter{theorem}       \medskip \noindent ({\bf \thetheorem}) }

\newcounter{minicapt}


%% file: macros-paper.tex
\newcommand{\ppm}[1]{\textcolor{red}{#1}}

\newcommand{\ppmx}[1]{}  
\newcommand{\ppmm}[1]{#1}  
\newcommand{\soutx}[1]{} 
\newcommand{\ssa}[1]{\textcolor{blue}{#1}}

\newcommand{\Bcat}{{\mathsf B}}  

\newcommand{\C}{{\mathbb{C}}}  
\newcommand{\R}{{\mathbb{R}}}
\newcommand{\N}{{\mathbb{N}}}  
\newcommand{\Z}{{\mathbb{Z}}}  
 
\newcommand{\AAAA}{\mathscr{A}}

\newcommand{\mqxx}[1]{}

\newcommand{\Sym}{\Sigma}  
\newcommand{\MATCH}{{\mathsf{Match}}}
\newcommand{\MAT}{{\mathsf{Mat}}} 
\newcommand{\MATCHA}{\MAT_{ACC}}
\newcommand{\MATCHa}{\mathsf{Matcha}}

\usepackage{faktor}\usepackage{amsmath}\usepackage{amssymb}
\usepackage[all]{xy}
\usepackage{tikz-cd}
\usepackage{float}
\makeatletter
\DeclareRobustCommand*{\mfaktor}[3][]
{
   { \mathpalette{\mfaktor@impl@}{{#1}{#2}{#3}} }
}
\newcommand*{\mfaktor@impl@}[2]{\mfaktor@impl#1#2}
\newcommand*{\mfaktor@impl}[4]{
   \settoheight{\faktor@zaehlerhoehe}{\ensuremath{#1#2{#3}}}%
   \settoheight{\faktor@nennerhoehe}{\ensuremath{#1#2{#4}}}%
      \raisebox{-0.5\faktor@zaehlerhoehe}{\ensuremath{#1#2{#3}}}%
      \mkern-4mu\diagdown\mkern-5mu%
      \raisebox{0.5\faktor@nennerhoehe}{\ensuremath{#1#2{#4}}}%
}
\makeatother

\newcommand{\U}[1]{\underline{#1}}
\newcommand{\NN}{\U{N}}

\newcommand{\beq}{\begin{equation}}
\newcommand{\eq}{\end{equation}}
\newcommand{\mat}{\begin{bmatrix}}
\newcommand{\tam}{\end{bmatrix}}

%% file: macros-paper2.tex
\newtheorem{theorem}{Theorem}[section]
\newtheorem{corollary}{Corollary}[theorem]
\newtheorem{lemma}[theorem]{Lemma}

\theoremstyle{definition}

\newcommand{\ff}{{\mathsf f}}
\newcommand{\Lamm}[2]{\Lambda^{#1}_{#2}}  
\newcommand{\lamm}{{\bf m}}

\newcommand{\tens}
{\otimes 
}

\newcommand{\ignore}[1]{}

%% file: CCwithglue00.tex
We first recall some notation 
as used in \cite{MartinRowell2021}, and some as  in Almateari's 
thesis and paper \cite{Almateari,Almateari26} 
(in turn taken from many sources, raging from \cite{adamek2004abstract} to \cite{Murnaghan}). 
This will also serve to fix some conventions. 

\medskip 

\newcommand{\asin}[1]{}
\newcommand{\nexx}{\medskip \\}

\noindent 
For $N,n \in \N$, we have the following notations.  
\medskip 
\\ 
The set 
$\NN = \{1,2,...,N \}$.
\nexx  
The set 
$
\Lamm{N}{n} \; =\; \{ \lamm \in \N_0^N \; : \;  \lamm.(1,1,...,1)\;  =n   \}
.$  
\nexx  
The function $\; \ff: \U{N}^n \rightarrow \Lamm{N}{n} \;$ is given 
by setting $\; \ff(w)_i \;=\; |\{ j | w_j = i \}|$. 
\nexx 
$\MAT$ is the monoidal category of matrices (over a fixed ground field - we can take it to be $\C$). Hence the object set is $\N$ and the monoidal product is multiplication on objects and Kronecker product on morphisms. 
\nexx 
$\MAT^N$ is the full monoidal subcategory of $\MAT$ generated by the object $N$; 
where the object $N^i$ is renamed as $i$, so the object monoid becomes $(\N_0,+)$ 
rather than $(N^{\N_0}, \times)$. 
Rows and columns of matrices in $\MAT^N(i,j)$ are labelled by $\NN^i$ and $\NN^j$ respectively\asin{MAT^N}. 
\nexx 
$\MATCH^N$ is the monoidal subcategory of $\MAT^N$ of CC matrices - matrices $R$ such that $R_{vw} \neq 0$ implies $\ff(v) = \ff(w)$\asin{monoisubcat}. 
\nexx 
$\MATCHA^N$ (sometimes written $\MATCHa^N$) is the monoidal subcategory of $\MAT^N$ of ACC matrices (additive CC matrices) - matrices $R$ such that $R_{vw} \neq 0$ 
implies $\Sigma_i w_i = \Sigma_i v_i$. 
(Here we refer to these categories but do not use them. 
See e.g. \cite{Almateari} for proof that this restriction on morphisms gives a monoidal category.)

%% file: CCwithglue0.tex
\ignore{{
\newcommand{\lav}{\lambda}
\newcommand{\mw}{\mu}
\newcommand{\off}{\overline{\ff}}
\newcommand{\unk}[2]{\raisebox{-5pt}{${#1 \atop #2}$}}
\newcommand{\unkk}[2]{\raisebox{-5pt}{$\overset{\scalebox{1.02}{$#1$}}{\scriptscriptstyle{#2}}$}} 
\newcommand{\unkkk}[2]{\raisebox{-4pt}{\begin{array}{c} #1 \\ {}^{#2}\end{array}}}
}}

\section{Charge-conservation with glue}  

In this section we generalise Almateari's methods from \cite{Almateari} to introduce a new strict monoidal 
category, again intended, 
like the categories $\MATCH^N$ and $\MATCHA^N$, 
as a target in representation theory.

\medskip 

\subsection{Introduction, motivation and preliminaries} 

Note that the categories  $\MATCH^N$ 
(from \cite{MartinRowell2021})
and $\MATCHA^N$ coincide for $N=2$, and that matrices of this type
do not exhaust the types seen in Hietarinta's rank-2 braid representation classification \cite{Hietarinta92}. 
A type much closer to including all the varieties in the classification is given by the `CCwg' matrices to be described below. 
As shown in Almateari's paper, 
the generic ansatz for a braid representation in ACC has a lot more parameters than CC. In this sense ACC is a big step up in difficulty from the solved CC case \cite{MartinRowell2021}, and indeed the advantageous calculus of CC does not have a directly similarly strong lift to ACC, so that after $N=3$ the all-higher-ranks ACC braid rep classification problem remains open. Like ACC, CCwg also extends CC, but a vestige of the CC calculus is retained, so it offers the twin advantages of both containing a more-complete transversal in rank-2; and more tools towards a potential all-rank solution. 
The analogues of non-semisimplicity are not direct, or even canonical, in higher representation theory, but CCwg can also be seen as a kind of generalisation to allow a form of higher non-semisimplicity, as we shall see.

\medskip 

In \S\ref{ss:constructCCwg} we construct the new category. 
In \S\ref{ss:glue-calculus} we aim to get as close as possible, in the 
(necessarily rigidified) 
higher representation theory setting, to the utility provided in the ordinary artinian non-semisimple representation theory setting by \ppmm{Wedderburn's Theorem} (see e.g. \cite{Cohn2003}), 
which uses the radical to bridge from non-semisimplicity to the Artin-Wedderburn Theorem and hence semisimple  structure. 

\ppmm{The utility of ideals (hence classical radicals) is limited in strict monoidal category representation theory, but in cases where all morphism sets are algebras we can bring ordinary ideal structure piecemeal ... We show in this sense that glue is radical ... }

\medskip \medskip

%% file: CCwithgluev2.tex
\newcommand{\MATCHwg}{\MATCH_g}
\newcommand{\rg}[1]{\textcolor{orange}{#1}}
\newcommand{\rgg}[1]{\textcolor{red}{#1}}

\subsection{The main construction: the categories $\MATCHwg^N$}
\label{ss:constructCCwg}
$\;$ 

Here we give the construction, summarised by the well-definedness theorem, for the CC-with-glue categories $\MATCHwg^N$ - in (\ref{th:CCwg}). 
We need various inputs, starting with an order on words.

\mdef 
Revlex order example (order in the case $\U{3}^4 \rightarrow \Lamm{3}{4}$):
\[
\begin{array}{cccccccccccccc}
w    &    1111 & 2111 & 3111 & 1211 & 2211 & 3211 & 1311 & 2311 & 3311 & 1121 & ... 
\\
\ff(w) &    \underline{400}  & \underline{310}  & \underline{301} &  310  & \underline{220}   & \underline{211} &  301 & 211 & \underline{202} & 310 & ...
\end{array}
\]
In the $\ff(w)$ row we have underlined the first instance of each composition in this order.

Observe that 
the 
(revlex) 
first instance of 
a word $w \in \NN^n$ with 
$\ff(w) = \; \lambda \; =\; \lambda_1 \lambda_2 ...\lambda_N$ is the orbit element with 
weakly decreasing $w_i$:
\beq  \label{eq:orbrep}
\overline{\ff}(\lambda) \;  := \; 
\underbrace{N,N,...,N}_{\lambda_N \; copies}, 
 \underbrace{N\!-\!1, N\!-\!1, ..., N\!-\!1}_{\lambda_{N\!-\!1}},...,
 \underbrace{11,,1}_{\lambda_1} 
\eq

\mdef   \label{de:Lammo}
Fix $N \in \N$. For $n \in \N$ give an order $(\Lamm{N}{n}, <)$
by ordering $\NN^n$ in revlex, apply $\ff$, then $\lambda < \mu$ if $\lambda$ first appears 
before $\mu$ in sequence $\ff(\NN^n)$. 
\\
(For later use we set $\lambda \in \Lamm{N}{n}$, $\mu \in \Lamm{N}{n'}$ 
incomparable unless $n=n'$.)

\mdef   \label{de:CCwg}
We say  a matrix $M \in \MAT^N$ is {\em rank-$N$ charge-conserving-with-glue} 
(CCwg)
if $\langle w | M | v \rangle \neq 0 $ implies $\ff(w) \leq \ff(v)$
(that is, $\ff(w) > \ff(v)$ implies $\langle w | M | v \rangle = 0 $).
\\
(By the incomparability convention only the zero non-square matrices are CCwg.)

\mdef Examples. Let us write $\MATCHwg^N(n,n)$ for the set of CCwg matrices. Then
\[
\begin{bmatrix} 1 & \rg1 & \rg1 & \rg1 \\ &1&1&\rg1 \\ &1&1&\rg1 \\&&&1\end{bmatrix}
\in \MATCHwg^2(2,2)
\]
\[
\begin{bmatrix} 
1 & \rg1 & \rg1 & \rg1 & \rg1 &\rg1 &\rg1 &\rg1 &\rg1 \\
  & 1 & \rg1 & 1 &\rg1&\rg1&\rg1&\rg1&\rg1 \\
  &   & 1 & &  \rg1& \rg1&1&\rg1&\rg1 \\
  & 1 & \rg1  &1&\rg1&\rg1&\rg1&\rg1&\rg1 \\
  &&&&1&\rg1&&\rg1&\rg1 \\
  &&&&&1&&1 &\rg1\\
  &&1&&\rg1&\rg1&1&\rg1&\rg1 \\
  &&&&&1&&1&\rg1 \\
  &&&&&&&&1
  \end{bmatrix}
\in \MATCHwg^3(2,2)
\]
(omitted entries are 0). 
Indeed the black entry positions 
correspond to a row $w$ and column $v$ that 
satisfy $\ff(w) = \ff(v)$; 
and the orange satisfy $\ff(w) < \ff(v)$. 
\\ 
Observe that we have a kind of partially-sorted upper-block-triangularity. 

\mdef  \label{de:glue}
In general we call the positions 
$\langle w | M | v \rangle  $
in a matrix $M \in \MAT^N(n,n)$
that satisfy $\ff(w) < \ff(v)$ {\em glue} positions. 
\\
We may call positions that satisfy $\ff(w) = \ff(v)$ {\em CC} 
(charge-conserving)
{positions}. 

\newcommand{\KKK}{\overline{K'}}

\mdef  \label{de:K'K}
For each $N,n,m$ 
we write $K': \MATCHwg^N(n,m) \rightarrow \MATCHwg^N(n,m)$ for the projection 
that sends all glue positions to zero. 
We call this the {\em de-gluing} map. 
We write $\KKK$ for the complementary projection that sends all non-glue 
(i.e. CC) positions to zero. 
Observe that $  M \mapsto K'(M)$ also gives a well-defined map  
$K: \MATCHwg^N(n,m) \rightarrow \MATCH^N(n,m)$. 
Of course all these maps are linear. 
\\ 
We write $\gamma$ for the equivalence relation on $ \MATCHwg^N(n,m)  $ 
given by $a \gamma b$ if $a-b$ is non-zero at most in glue positions. 
Thus  $\MATCH^N(n,m) $ is a transversal of the classes of $\gamma$. 

\medskip

To prove our main theorem 
\ref{th:CCwg}
below we will  need to understand the order 
$(\Lamm{N}{n},>)$ a bit better. 

\ignore{{
\newcommand{\lav}{\lambda}
\newcommand{\mw}{\mu}
\newcommand{\off}{\overline{\ff}}
\newcommand{\unk}[2]{\raisebox{-5pt}{${#1 \atop #2}$}}
\newcommand{\unkk}[2]{\raisebox{-5pt}{$\overset{\scalebox{1.02}{$#1$}}{\scriptscriptstyle{#2}}$}} 
\newcommand{\unkkk}[2]{\raisebox{-4pt}{\begin{array}{c} #1 \\ {}^{#2}\end{array}}}
}}

\mdef   \label{de:Lammoo}
Define {\em an} order  $(\Lamm{N}{n},\prec)$  
by $\lav \prec  \mw$ 
(i.e. sequence $\lav_1 \lav_2 ... \lav_n \;\prec\; \mw_1 \mw_2 ... \mw_n$)
if the first difference from
the left is, for some $i$, $\lav_i \neq \mw_i$, and then $\lav_i > \mw_i$.

\mdef 
Example: 
Consider $\lambda = 53104$ and $\mu=53203 \in \Lamm{5}{13}$.
The first difference is $\lambda_3 = 1 < 2= \mu_3$, so $\lambda \succ \mu$.

To compare with $(\Lamm{N}{n},<)$: 
With $\NN^n$ in revlex order, the first  
instance of a word $w$ with $\ff(w)=$ 53104 $>$ (meaning first instance comes later in revlex order than)  the first instance of a word with $\ff(v)= 53203$.
For 53104 the first instance is the word 
$\off(53104) = 5555322211111$, by (\ref{eq:orbrep}).
For 53203 it is 5553322211111. 
Starting at 5553322211111 we count up, first to 
1114322211111, then 
\\ 2114322211111,
..., 5554322211111, 1115322211111, ..., 5555322211111.
\\
Observe that since early digits change quickest in revlex the right-hand end of 
the word (here 322211111) does not change in this sequence. 
The immediately preceding letter starts at 3 and gradually ascends to 5.

\mdef {\bf Proposition}. \label{pr:dagger}
Fix $N,n\in\N$. 
The order $(\Lamm{N}{n},\prec)$ from (\ref{de:Lammoo}) is our order 
$(\Lamm{N}{n}, <)$
from (\ref{de:Lammo}).

\medskip

\noindent {\em Proof}. 
Put $\NN^n$ in revlex order. 
Keep in mind (\ref{eq:orbrep}).
So suppose $\lambda \prec \mu$, and compare words $ \off(\lambda) , \off(\mu) $. 
The compositions are the same up to $\lambda_{i-1}$ for some $i$ 
(and then $\lambda_i > \mu_i$) 
so both words 
end in the same descending pattern 
$$
\underbrace{i,i,\ldots,i}_{\mu_i},
\underbrace{i-1,i-1,\ldots,i-1}_{\lambda_{i-1}},i-2,\ldots,
\underbrace{1,1,\ldots,1}_{\lambda_1}.
$$
Note that the immediately preceding letter is $i$ for  $\off(\lambda)$ and $>i$ for $\mu$
(as in the 53104 example). But of course here this digit is gradually ascending in revlex, so $\lambda<\mu$.
\qed 



\mdef 
For $v \in \NN^n$ we set $\#_i(v) = \ff(v)_i$, the number of occurrences of symbol $i$ in $v$.  
$\;$ 
For example $\#_2 (1121) = \ff(1121)_2 =  1$.

\mdef 
Consider  $v =  v_1 v_2 \cdots v_n v_{n+1} \cdots v_{n+m}
\in \NN^{n+m}$.  
We denote by  
${\unkk{v}{n-}} \in \NN^n$ the subword that is the first $n$ symbols of $v$, 
and by 
$\unkk{v}{-m} \in \NN^m $ the last $m$ symbols,
so that 
$\unkk{v}{n-} \unkk{v}{-m} =v$.

\mdef {\bf Proposition}. \label{pr:ddagger}
Consider  $v,w \in \NN^{n+m}$.   
\\ 
(I) If $\ff(v) < \ff(w)$ then  
\ignore{{
$$
\ff(v_1, v_2, ..., v_n, v_{n+1},..., v_{n+m})
\; < \;  
\ff( w_1, w_2, ...,w_n,w_{n+1},...,w_{n+m} )
$$
%
We have 
${\unkk{v}{n-}}, 
{\unkk{w}{n-}} \in \NN^n$,
$\unkk{v}{-m},\unkk{w}{-m} \in \NN^m $,
the indicated 
projections to subwords 
so that $\unkk{v}{n-} \unkk{v}{-m} =v$ (resp. for $w$).
}}
either $\ff(\unkk{v}{n-}) < \ff(\unkk{w}{n-})$ or 
$\ff(\unkk{v}{-m})<\ff(\unkk{w}{-m})$, or both. 
\\
(II) If $\ff(v) = \ff(w)$ then either 
$\ff(\unkk{v}{n-}) = \ff(\unkk{w}{n-})$ {\em and} 
$\; \ff(\unkk{v}{-m}) = \ff(\unkk{w}{-m})$ 
or both are inequalities and in opposite directions.
\\
(III) If $\ff(v) > \ff(w)$ then at least one of 
$\ff(\unkk{v}{n-}) > \ff(\unkk{w}{n-})$  and 
$\; \ff(\unkk{v}{-m}) > \ff(\unkk{w}{-m})$ 
holds. 

\medskip  

\noindent 
{\em Proof}. 
Of course 
$\ff(\unkk{v}{n-}) +\ff(\unkk{v}{-m}) = \ff(v)$ for any $v \in \NN^{n+m}$.
\\
(I)
Suppose for a contradiction that both  
$\ff(\unkk{v}{n-}) \geq \ff(\unkk{w}{n-})$  and 
$\ff(\unkk{v}{-m}) \geq \ff(\unkk{w}{-m})$.
The case both equal gives $\ff(v)=\ff(w)$ --- a contradiction.
So then 
by (\ref{de:Lammoo},\ref{pr:dagger})
there is a lowest $i$ so that $\#_i (\unkk{v}{.-.}) > \#_i(\unkk{w}{.-.})$
on at least one `side', i.e. for $.\!-\!.$ is $n-$ or $-m$.
So then 

\beq  \label{eq:hashsquarej}
\begin{array}{ccccc} 
\#_j (\unkk{v}{n-}) &+ \#_j(\unkk{v}{-m}) =&\#_j(v)    \\
|| & || & || \\
\#_j (\unkk{w}{n-}) &+ \#_j(\unkk{w}{-m}) =&\#_j(w)
\end{array}
\eq 
for $j <i$, and (with at least one $\leq$ strict)
$$
\begin{array}{ccccc} 
\#_i (\unkk{v}{n-}) &+ \#_i(\unkk{v}{-m}) =&\#_i(v)    \\
{\mathsf{V}}| & {\mathsf V}| & {\mathsf V} \\
\#_i (\unkk{w}{n-}) &+ \#_i(\unkk{w}{-m}) =&\#_i(w)
\end{array}
$$
--- a contradiction, by the $\prec$ formulation of $(\Lamm{N}{n},<)$ 
allowed by (\ref{pr:dagger}).
\ignore{{
...
Then these \ppm{ / }
the orbit representatives as in (\ref{eq:orbrep})
 are the same up to $v_{i-1}$ for some $i$ (possibly $i=1$) and then 
$v_i > w_i$.
Now break them both into lengths $n|m$.
\\
If $i\leq n $ then $v_1..v_n < w_1 .. w_n$.
\\
If $i>n$ then $v_{n+1} ..v_i v_{i+1} ..v_{n+m}$ 
is first different from $w_{n+1} ..w_{n+m}$ at $v_i > w_i$,
so $v_{n+1}..v_{n+m} < w_{n+1}..w_{n+m}$.
So, altogether, at least one subword obeys $v_{[]} < w_{[]}$.
\ppm{[-finish this!]}
}}
\\
(II) Suppose $\ff(\unkk{v}{n-}) =\ff(\unkk{w}{n-})$. 
Then the outer two vertical equalities hold in (\ref{eq:hashsquarej}) for all $j$;
and hence the inner equality holds. 
Similarly if   $\ff(\unkk{v}{n-}) > \; \ff(\unkk{w}{n-})$ then there is a lowest $i$ 
where  $\#_i(\unkk{v}{n-} ) < \; \; \#_i(\unk{w}{n-})$, and then  
$\#_i(v) = \#_i(w)$ forces 
$\#_i(\unkk{v}{-m} ) \; > \; \; \#_i(\unkk{w}{-m})$. 
And similarly with the inequalities reversed. 
\\
(III) Similar argument to (I). 
\qed

\medskip 

\mdef {\bf{Theorem}}.  \label{th:CCwg} 
Fix $N \in \N$. 
The subset of CCwg morphisms of $\MAT^N$ gives a strict monoidal subcategory
(hence denoted $\MATCHwg^N$). 

\medskip 

\noindent 
{\em Proof}. 
We require to show closure under the category and monoidal compositions
(observe that the units are clearly present). 
\\ 
For the category composition, 
consider $L,M \in \MAT^N$,  
composable, and in the CCwg subset.
If either is not square then $LM=0$ so there is nothing to check.
So take $L,M \in \MAT^N(n,n)$.
For $w,v \in \N^n$ 
\beq  \label{eq:wLMv}  
\langle w | LM | v \rangle \; 
= \;\; \sum_{u \in \N^n} \langle w | L | u \rangle \; \langle u | M | v \rangle 
\eq 
We require to show that $\ff(w) \not\leq \ff(v)$ (i.e. $\ff(w) > \ff(v)$) implies
$  \langle w | LM | v \rangle =0 $.
Suppose $\ff(w) > \ff(v)$. We have   
\[
\langle w | LM | v \rangle \; 
= \; \sum_{u \in \N^n \over \ff(u)<\ff(w)} 
        \overbrace{\langle w | L | u \rangle}^{=0} \langle u | M | v \rangle 
  + \sum_{u \in \N^n \over \ff(u) \geq \ff(w)} 
      \langle w | L | u \rangle \overbrace{\langle u | M | v \rangle}^{=0}
      \;\; = \; 0
\]
as required, where the first indicated factor is 0 by the CCwg condition on $L$;
and the second because $\ff(u) \geq \ff(w) >\ff(v)$ implies $\ff(u) > \ff(v)$ here. 
\\
For the monoidal composition, take $L \in \MAT^N(n,n')$, $M\in\MAT^N(m,m')$.
If $n\neq n'$ or $m\neq m'$ then $L\otimes M=0$ so there is nothing to check, 
so take 
$n=n'$, $m=m'$. 
For $w,v\in \NN^{n+m}$
we require to show $\langle w | L \otimes M |v\rangle =0$ if $\ff(w)>\ff(v)$.
Suppose $\ff(w) > \ff(v)$.
Recall we have 
\beq  \label{eq:kronekerid} 
\langle  {{\scalebox{1.02}{$w$}}}   | L \otimes M |      {v} \rangle  
\; = \; 
  \langle \unkk{w}{n-} | L  |\unkk{v}{n-}\rangle  \;   \langle \unkk{w}{-m} |  M | \unkk{v}{-m}\rangle
\eq 
It is enough to show   
one of 
$\ff( \unkk{w}{n-} ) > \ff(\unkk{v}{n-})$ or  
$\ff( \unkk{w}{-m} ) > \ff(\unkk{v}{-m})$
given $\ff(w) > \ff(v)$.
\ignore{{
For this we need to understand a property of the $(\Lamm{N}{n},>)$ order as 
in (\ref{pr:dagger}) and (\ref{pr:ddagger}).


\mdef (Back to the proof of (\ref{th:CCwg})) 
}}
We indeed have this from (\ref{pr:ddagger}), so that one or other factor in (\ref{eq:kronekerid}) vanishes here, hence the entry vanishes as required.
\qed

\mdef {\bf Proposition}. 
Let $N \in \N$. The category $\MATCH^N$ is a monoidal subcategory of $\MATCHwg^N$. 
\medskip

\noindent 
{\em Proof}. 
The CC condition defining $ \MATCH^N$ can be expressed as $R_{vw} \neq 0$ implies $\ff(v) = \ff(w)$.
\qed

\mdef    \label{pa:2radical}
Note the following from (\ref{pr:ddagger}) and the formulation in the proof 
of Thm.\ref{th:CCwg}
above. 
\medskip 

\noindent 
{\bf Lemma}. 
Let 
$L \in \MATCHwg^N(n,n)$, $M\in\MATCHwg^N(m,m)$, 
and $w,v \in \NN^{n+m}$. 
\medskip 
\\
(A) If $\ff(w) =\ff(v)$ then 
\\
(I) 
if $\ff(\unkk{w}{n-}) = \ff(\unkk{v}{n-})$ then 
$\ff(\unkk{w}{-m}) = \ff(\unkk{v}{-m})$ and both factors in  
$ \langle w | L \otimes M |v\rangle $  
from (\ref{eq:kronekerid})
come from 
`CC entries' in $L$ and $M$. 
$\;$ 
On the other hand 
\\
(II) 
if $\ff(\unkk{w}{n-}) \neq  \ff(\unkk{v}{n-})$
then one of the factors is `glue', and the other is zero.

\medskip 

\noindent (B)
And on the other hand if $\ff(w) < \ff(v)$ (i.e. we are looking at a matrix entry 
$ \langle w | L \otimes M |v\rangle $ 
in a glue position) then at least one factor 
comes from a `glue entry' in $L$ or $M$. 
\qed 

\medskip

\subsection{Towards a glue calculus for representation theory}  
\label{ss:glue-calculus}
$\;$ 

Let us now consider the generic glue ansatz for an $R$-matrix
- extending the generic charge-conserving ansatz solved to classify braid representations $F:\Bcat\rightarrow\MATCH^N$ in \cite{MartinRowell2021}. 
(Or indeed for representations of generators in any other such strict monoidal category.)
We want to establish properties of the constraints that depend on whether a parameter is glue (i.e. allowed non-zero in CCgw but not allowed in CC) or non-glue. 

For example 
\beq    \label{eq:Morange}
D=\begin{bmatrix}
    \alpha & \rg \beta \\ & \gamma
\end{bmatrix}  \in \MATCHwg^2(1,1),  \hspace{1cm}
M=
\begin{bmatrix} 
a & \rg b & \rg c & \rg d \\ 
&f&g&\rg h \\ 
&k&l&\rg m \\
&&& r
\end{bmatrix}
\in \MATCHwg^2(2,2)
\eq
has the glue coloured \rg{orange}.

The basics:
\[
\langle 11 | M  = [1,0,0,0]M = [a,b,c,d] ,
\hspace{1cm} 
\langle 22 | M  = [0,0,0,1]M = [0,0,0,r] 
\]
\[
\langle 22 | M | 12 \rangle \; = \; [0,0,0,1] \; M\begin{bmatrix}
    0 \\ 0 \\ 1 \\ 0
\end{bmatrix} = 0 
\]
- consistent with  
$\ff(22) > \ff(12) $ 
as in (\ref{de:CCwg}). 

Meanwhile for example
\[
D \otimes D = \begin{bmatrix}
    \alpha \alpha & \alpha \rg\beta & \rg\beta \alpha & \rg{\beta\beta} \\
                  & \alpha\gamma   &           &  \rg\beta \gamma \\
                  &                & \gamma\alpha &\gamma\rg\beta \\
                  &             &                & \gamma\gamma 
\end{bmatrix}
\]
\[
M\otimes M \; = \; \left[ \begin{array}{cccc|cccc|cccc|cccc}
    aa&a\rg b&a \rg c&a\rg d &\rg{b}a &\rg{bb} &\rg{bc}&\rg{bd} & ... \\
      & af   & ag    &a\rg{h}&        &\rg{b}f &\rg{b}g&\rg{bh} &  ...\\
      & ak   & al    &a\rg{m}&        &\rg{b}k &\rg{b}l&\rg{bm} & ... \\
      &&&ar&     &&&\rg{b}r & ...  \\ \hline 
      &&&&fa& f\rg{b} &f\rg c& f\rg d& ... \\
      &&&&&ff& fg&f\rg h& ... \\
      &&&&&fk&fl&f\rg m& ... \\
      &&&&&&&fr& ... \\ \hline
      &&&&ka&k\rg b&k\rg c&k\rg d& ... \\ 
      &&&&&kf&kg  &k\rg h& ... \\
      &&&&&kk&kl&k\rg m& ... \\
      &&&&&&& kr & ... \\ \hline 
      &&&&&&&& \ddots
\end{array}
\right]
\]

\newcommand{\unkkkk}[2]{
\begin{array}{c} #1 \\ {#2}\end{array}
}

\mdef 
Let $L,M \in \MATCHwg^N(n,n)$ and 
consider matrix entries 
$
\langle w| LM |v\rangle  .
$
When $\ff(w)=\ff(v)$ we have 
\[
\langle w | LM | v \rangle \; 
= \; \sum_{\unkkkk{u \in \N^n }{ \ff(u)<\ff(w)}} 
        \overbrace{\langle w | L | u \rangle}^{=0} \langle u | M | v \rangle 
  + \sum_{\unkkkk{u \in \N^n }{ \ff(u) = \ff(w)}} 
      \langle w | L | u \rangle 
      \langle u | M | v \rangle
\] \[  \hspace{1in} 
  + \sum_{\unkkkk{u \in \N^n }{ \ff(u) > \ff(w)}} 
      \langle w | L | u \rangle 
      \overbrace{
      \langle u | M | v \rangle
      }^{=0}
\]
from which we see that we get contributions only from the strictly CC entries in $L,M$. 
We deduce the following. 


\mdef  {\bf Proposition}.  \label{pr:id-post-glue}
Consider a given set of rank-$N$ CCwg matrices in level-$n$. 
Write $M_1, M_2, ... \in \MATCHwg^N(n,n)$ for these matrices. 
Any identity expressed in terms of the matrices  
$\{ M_i \}$ 
(such as $M_1 M_2 M_1 = M_2 M_1 M_2$) 
and obeyed by
them is also obeyed by the matrices 
$K(M_i)$
obtained by projecting out the glue. 
\qed 

\mdef    \label{le:w<v}
Now consider $\ff(w) < \ff(v)$ - we have
\[
\langle w | LM | v \rangle \; 
= \; \sum_{\unkkkk{u \in \N^n }{ \ff(u)<\ff(w)}} 
        \overbrace{\langle w | L | u \rangle}^{=0} \langle u | M | v \rangle 
  + \sum_{\unkkkk{u \in \N^n }{ \ff(u) = \ff(w)}} 
      \langle w | L | u \rangle 
      \rg{\langle u | M | v \rangle}
\] \[  \hspace{1in} 
  + \sum_{\unkkkk{u \in \N^n }{ \ff(u) > \ff(w)}} 
      \rg{\langle w | L | u \rangle} 
      \langle u | M | v \rangle
\]
- we see that every such entry contains a contribution with a glue factor 
(coloured \rg{orange})
from $L$ or $M$ or both. 

We deduce the following. 

\newcommand{\Alg}{{\mathsf A}}

\mdef  \label{cth:HoG}
{\bf{Theorem}}. [HoG Theorem.]
(I) Glue is a nilpotent ideal 
- meaning that for each $n$ the glue subset of the algebra $\MATCHwg^N(n,n)$ is a nilpotent ideal. 
Thus for $N,n>1$ the algebra  $\MATCHwg^N(n,n)$  is non-semisimple. 
And furthermore any subalgebra $M$ with 
$dim(M)>dim(K(M))$
(as in (\ref{de:K'K}))
is non-semisimple.
\\
(II) Also if $\rho: \Alg \rightarrow \MATCHwg^N(n,n)$ is a representation of an algebra $\Alg$ then so is 
the de-gluing 
$K\circ \rho$ 
(again as in (\ref{de:K'K})); 
and the image is 
a quotient by {\em part (or all) of} the radical. 
In particular the irreducible factors of $\rho(\Alg)$ and $(K\circ\rho)(\Alg)$ are the same. 
 \medskip   \\ 
\ppmm{{\em Proof}.  (I) We  see from the final term in the expansion in (\ref{le:w<v}) that products of glue elements are glue elements between indices of greater separation in the order (i.e. $\ff(v)>\ff(u)>\ff(w)$). But for any finite $n$ the ordered set is finite, so, iterating, all products with sufficiently many factors vanish. 
\\ 
(II) From (\ref{pr:id-post-glue}), or directly, we have that 
$(K\circ\rho)(ab) = K(\rho(a) \rho(b)) = \; (K\circ\rho)(a) . (K\circ\rho)(b)$. 
\qed }

\mdef Remark. Observe that a representation of an algebra that factors surjectively through a semisimple algebra need not be semisimple (since the radical could be in the kernel). But a rep that factors surjectively through a non-semisimple algebra is necessarily non-semisimple. 

\medskip 

\mdef Next we can turn to the monoidal product. 
In a nutshell, this part of the calculus can be pulled through from (\ref{pa:2radical}). 

\medskip 

There are some questions as to how to present this as an analogue of Jacobson radicals - questions arising from the non-canonical-ness of higher representation theory. 
So rather than exhibiting a grand theory here we will 
limit ourselves to an example application. 

We have shown, for example, 
the following.

\mdef {\bf Corollary}.   \label{co:F_K}
Every braid representation of form 
$F: \Bcat \rightarrow \MATCHwg^N$ `restricts' (cf. e.g. the Wedderburn--Malcev Theorem) 
to a charge-conserving representation. 
That is, writing $\sigma$ for the elementary braid in $B_2$ monoidally generating $\Bcat$, 
so 
$R=F(\sigma)$ 
is the $R$-matrix that 
determines $F$; then 
$K(R)$ is an $R$-matrix determining a charge-conserving representation $F_K$. 
And furthermore the simple content of the ordinary representations $F(B_n)$ and $F_K(B_n)$ is the same for all $n$. 
\ignore{{ 
Consider the functors 
$$
I: \MATCH^N \rightarrow \MATCHwg^N
$$ 
(by inclusion) 
and 
$$
J: \MATCHwg^N \rightarrow \MATCH^N
$$ 
(kill glue entries). 

\ppm{[to finish]}
}}

\medskip 

Next we turn to examples and applications.

%% file: tex/N2.tex
\subsection{Preamble: Braid representations in rank $N=2$}

It will be useful here to recall Hietarinta's classification of varieties of braid representations in rank-2 
(up to a certain notion of equivalence - see later) \cite{Hietarinta92}. 
Excluding the trivial representation Hietarinta effectively gives the following ten cases. 
\input{tex/HietarintaN2}
It will be convenient also to have names for the individual varieties here. 
We take most of the names from those used 
in the CC classification \cite{MartinRowell2021}, informed now by the glue machinery. 
The varieties $\Rslash, R_f, R_a$  
are CC on the nose. 
Thus the first is slash type. 
The case $R_a$ 
is a-type; and the 
case $R_f$ 
is f-type - respectively corresponding to representations with $U_qgl(1|1)$ and $U_qsl(2)$ symmetry, 
but the labels a and f refer to {\em antiferromagnetic} and {\em ferromagnetic}, coming from the complete-graph 
characterisation  
of CC representations.  
The 
representation $\Raslash$ above is antislash. 
\ignore{{While $\Raslash$ is not locally equivalent to a CC case, \emph{it is} $\infty$-equivalent, in fact DS equivalent to a specialisation of $\Rslash$.  Indeed one finds that $T\otimes T$ with $T=\begin{bmatrix}
    0 & 1\\1&0
\end{bmatrix}$ commutes with the specialisation $\Rslash(s=k)$ with $s=k$ in $\Rslash$, and hence $T\otimes I\Rslash(s=k) (T\otimes I)^-1$ is another YBO.  We see that this is $\Raslash$.  Now DS equivalence implies $\infty$-equivalence via $T_n=I\otimes T\otimes T^2\cdots \otimes T^{n-1} $ (see \cite{MRT2}).
}}
The  
cases $R_{fI}, R_{fII}, R_{fIII}$ are  f-glue-I to III, since they are all glue representations that project/de-glue onto various parts of the f variety. 
In fact they do so at points where the f variety intersects the slash variety, so there is choice in the names here.

The 
case $R_{ag}$
is `a-glue', by the same logic. 
The last two we call Ising and 8-vertex respectively. These two do not fit in the glue scheme. 

\mdef 
Note that 
while $\Raslash$ is not locally equivalent to a CC case, \emph{it is} $\infty$-equivalent, in fact DS equivalent to a specialisation of $\Rslash$.  Indeed one finds that $T\otimes T$ with $T=\begin{bmatrix}
    0 & 1\\1&0
\end{bmatrix}$ commutes with the specialisation $\Rslash^{s=k}$ with $s=k$ in $\Rslash$, and hence $(T\otimes I)\Rslash^{s=k} (T\otimes I)^{-1}$ is another YBO.  We see that this is $\Raslash$.  Now DS equivalence implies $\infty$-equivalence via $T_n=I\otimes T\otimes T^2\cdots \otimes T^{n-1} $ 
(see \cite{MRT25}).

\mdef A computer calculation \cite{Maple} shows that the 8 vertex case is $3$-equivalent to a CC case--the $R_a$ case with $2$ eigenvalues of multiplicity $2$.  Does this extend to an $\infty$-equivalence?

\medskip

Next we consider specifically how the varieties above are directly related, i.e. their intersections. 

\mdef 
Observe that if we put $k=1$ and $p=q=-1$ then $R_{f}$ collapses to the de-glued version of $R_{fI}$. 
If  we put  $p=q=k$ then $R_{f}$ collapses to the de-glued version (or the glue free subvariety) of $R_{fII}$ up to overall scalar; and 
to the de-glued/glue-free $R_{fIII}$ on the nose. 
(Caveat: variable names $p,q$ and so on are not coordinated at these intersections.)

\medskip 
Under X-symmetry we can transform $R_a$ to 
\[  R_a' =   \left[ \begin{array}{cccc} k^2 \\ & 0& k^2 \\ & pq & k^2-pq \\ &&& -pq\end{array} \right]  \]
which depends on $p,q$ only through $pq$, and only on $k$ through $k^2$. 
If we put $k^2 = p'$ and $pq=q'$ then $R_a$ collapses to the de-glued $R_{ag}$ 
(with parameters disambiguated by replacing the parameters in $R_{ag}$ with primed versions). 


\mdef 
With extension to higher ranks in mind, 
it is an intriguing question what a-glue looks like in the sense of solutions to the glue ansatz extending the original formulation of $R_a$. And indeed through the orbit of the X-symmetry. 
For now it will be enough to consider the part of the orbit given by 
\beq   \label{eq:Rag} 
R_{ag}^x = 
\left[\begin{array}{cccc}
a  & p  & q  & r  
\\
 0 & 0 & a x  & s  
\\
 0 & -\frac{b}{x} & a +b  & t  
\\
 0 & 0 & 0 & b  
\end{array}\right]
\eq 
(enough in the sense that X symmetry does not change the CC part in the 11 and 22 channels). 
Note that this covers $R_a$ by putting $a=k^2$ and $b=-pq$ and then $x=q/k$; and $R_a'$ by putting $x=1$. 

From the anomaly vanishing (and invertibility) we find firstly 
$$
(px+q)(a+b)=0 , \hspace{.3cm} 
(px+q)(ax+b)=0 , \hspace{.3cm} 
(s-tx)(a+b)=0, \hspace{.3cm} 
(s-tx)(ax-b)=0
$$ 
Thus for example if $x \neq 1$ then $q=-px$ and $t=s/x$. 
Putting $q=-px$ and $t=s/x$ we obtain 
\[
s (ax+b) = px   (ax-b)   
\]
so, if $(ax+b) \neq 0$ then $q,s,t$ all depend linearly on $p$. 
Imposing the corresponding identity for $s$ then the anomaly vanishes if 
$
(ax - b) (p^2 x (x^2a -b)   +   r(ax  + b)^2) =0  .
$
Since we assumed $ b \neq -ax$  this determines $r$ linearly in terms of $p^2$, 
unless $b=ax$. 
We have the braid representation
\[
\left[\begin{array}{cccc}
a  & p  & -px  & p^2 x \frac{x^2 a-b}{(ax+b)^2}  
\\
 0 & 0 & a x  & px\frac{ax-b}{ax+b}  
\\
 0 & -\frac{b}{x} & a +b  &   p\frac{ax-b}{ax+b}
\\
 0 & 0 & 0 & b  
\end{array}\right]
\]

On the other hand suppose $x=1$. Then $t=s$ and $q=-p$. 
If $b=a$ then $s=0$ and we get the solution
\[
\left[\begin{array}{cccc}
a  & p  & -p  & r  
\\
 0 & 0 & a   & 0  
\\
 0 & -{a} & 2a   &   0
\\
 0 & 0 & 0 & a  
\end{array}\right]
\]
which is locally equivalent to the equal-eigenvalue case of $R_{ag}$.
If $b \neq a$,  
the analysis depends on whether $b=-a$. 
Let us continue, for now, with a more systematic (if partly less transparent) scheme.

\mdef\label{Xsym analysis} 
We can develop an ansatz for higher ranks, starting with the set up of $R_{ag}^x$ of \eqref{eq:Rag}.  
One readily verifies that 
\[\left[ \begin {array}{cccc} a&0&0&0\\ 0&0&xa&0
\\ 0&-{\frac {b}{x}}&a+b&0\\ 0&0&0
&b\end {array} \right],\left[ \begin {array}{cccc} {\frac {b \left( -t
x+q \right) }{x \left( tx+q \right) }}&-{\frac {q}{x}}&q&-{\frac {-{t}
^{2}{x}^{3}-{t}^{2}{x}^{2}+{q}^{2}x-2\,tqx-{q}^{2}}{4\,xb}}
\\ 0&0&{\frac {b \left( -tx+q \right) }{tx+q}}&tx
\\ 0&-{\frac {b}{x}}&{\frac {b \left( -tx+q \right) 
}{x \left( tx+q \right) }}+b&t\\ 0&0&0&b\end {array}
 \right] \]  \[ \left[ \begin {array}{cccc} a&0&0&r\\ \noalign{\medskip}0&0&-b&-{
\frac {bt}{a}}\\ \noalign{\medskip}0&a&a+b&t\\ \noalign{\medskip}0&0&0
&b\end {array} \right]
,  \left[ \begin {array}{cccc} a&-{\frac {aq}{b}}&q&r
\\ \noalign{\medskip}0&0&b&0\\ \noalign{\medskip}0&-a&a+b&0
\\ \noalign{\medskip}0&0&0&b\end {array} \right] 
\] 
are all R-matrices. 
Using Maple \cite{Maple} on Gr\"obner bases and solving the braid anomaly, we 
find that this list is exhaustive. 

The first solution is already CC, while the second may be gauged to a CC form, by conjugating by $T\otimes T$ where $ T=\left[ \begin {array}{cc} 1&-{\frac {tx+q}{2\,b}}
\\ \noalign{\medskip}0&1\end {array} \right].
$

The third and fourth solutions correspond to a special choice of $x$, namely $x=-a/b$ and $x=a/b$.  These can be gauged, by means of a matrix of the form  $\left[ \begin {array}{cc} 1&y
\\ 0&1\end {array} \right]$ to the following: 

\[ \left[ \begin {array}{cccc} a&0&0&{\frac {4\,{a}^{2}r+a{t}^{2}-b{t}^{
2}}{4\,{a}^{2}}}\\ 0&0&-b&0\\ 0&a&
a+b&0\\ 0&0&0&b\end {array} \right] , \left[ \begin {array}{cccc} a&0&0&-{\frac {a{q}^{2}-4\,{b}^{2}r-b{q}^
{2}}{4\,{b}^{2}}}\\ 0&0&b&0\\ 0&-a
&a+b&0\\ 0&0&0&b\end {array} \right].
\]
Since $r,t,q$ are arbitrary real numbers, we may choose $t=q=0$ to obtain the solutions of the form $R_{ag}$. In fact, as long as the $(1,4)$ entry is non-zero, we may gauge that entry to be $1$.  Indeed, here the gauge choice is a diagonal matrix. Moreover, having done this, we can make the two solutions themselves gauge equivalent (simply changing the sign) and the $a\leftrightarrow b$ symmetry then yields the single non-CC representative 

\[R_{a,b}:=\left[ \begin {array}{cccc} a&0&0&1\\ 0&0&-b&0\\ 0&a
&a+b&0\\ 0&0&0&b\end {array} \right].\]

\mdef When $a\neq b$, the above solution $R_{a,b}$ is diagonalisable.  The representation theory becomes more subtle when $a=b$, in which case the Jordan form has 2 $2\times 2$ blocks.  For the purposes of analysing this case it is also convenient (and possibly sufficiently generic) to set $a=b=1$--the unipotent case.
\subsection{The parity property} 

It will be useful in representation theory to exploit a simple and purely mechanical application of the machinery in 
\S\ref{ss:constructCCwg}. 
Observe that while for the CC construction it is necessary only that the symbols labelling rows and columns 
in $\MATCH^N(1,1)$ are distinct - i.e. the construction is symmetrical under the action of $\Sym_N$ permuting these labels, in $\MATCH_g^N$ the symbols must at least be totally ordered. 
(In a sense this was already true  in $\MATCH^N$ since we order the basis before writing out matrices; but the construction overall has the $\Sym_N$ symmetry.) 
This can be compared with $\MATCHA^N$, where the construction even makes use of the numerical values of the symbols $\{ 1,2,...,N\}$. 
Now let us say that two words in $\underline{N}^n$ have the same {\em parity} if their sums have the same parity (for example 123 and 222 have the same parity; and 11 and 22 have the same parity; but 11 and 12 do not).  

\mdef  \label{de:parp}
Fix $N$. Say a matrix $M$ in $\MAT^N$ is parity preserving if $M_{vw}\neq 0$ implies 
that words $v,w$ have the same parity. 
\\
Examples: 
\\
(I) CC matrices are parity preserving. 
\\ 
(II) every matrix in Hietarinta's rank-2 transversal (recalled above) is parity preserving except for $R_{fII}$ and $R_{fIII}$. 

\mdef  \label{lem:parity}
{\bf{Proposition}}.  
The product of parity preserving matrices is parity preserving. 
The Kronecker product of parity preserving matrices is parity preserving. 
\medskip 
\\
{\em Proof}. Consider composable matrices $L,M$ as in (\ref{eq:wLMv}). 
Suppose words $w,v$ have different parities. Then for each summand either $u$ has different parity to $w$, or to $v$, 
so the contribution to the product is 0, as required.
Now consider the Kronecker, as in (\ref{eq:kronekerid}). 
Suppose WLOG that $w$ is odd and $v$ even. 
If the $n$ part of $w$ is odd then (the $m$ part is even and) if the $n$ part of $v$ is even the first factor
vanishes and there is nothing to show, so suppose the $n$ part of $v$ is odd - then the $m$ part is odd and the second factor vanishes.
\qed

\medskip 

We will apply this repeatedly in what follows.  See for example \S\ref{ss:casediamond}.

%% file: tex/HietarintaN2.tex

\newcommand{\Hslashglue}{
\left[ \begin {array}{cccc} k&q&p&s\\ \noalign{\medskip}0&0&k&q
\\ \noalign{\medskip}0&k&0&p\\ \noalign{\medskip}0&0&0&k\end {array}
 \right] 
 }
\newcommand{\Hslash}{
 \left[ \begin {array}{cccc} k&0&0&0\\ \noalign{\medskip}0&0&p&0
\\ \noalign{\medskip}0&q&0&0\\ \noalign{\medskip}0&0&0&s\end {array}
 \right] 
 }

\newcommand{\Haglue}{
 \left[ \begin {array}{cccc} p&0&0&k\\ \noalign{\medskip}0&0&p&0
\\ \noalign{\medskip}0&q&p-q&0\\ \noalign{\medskip}0&0&0&-q\end {array}
 \right]
 }
 
\newcommand{\HslashDS}{
 \left[ \begin {array}{cccc} 0&0&0&p\\ \noalign{\medskip}0&k&0&0
\\ \noalign{\medskip}0&0&k&0\\ \noalign{\medskip}q&0&0&0\end {array}
 \right] 
}

\newcommand{\Hslashgluex}{
 \left[ \begin {array}{cccc} {k}^{2}&-kp&kp&pq\\ \noalign{\medskip}0&0
&{k}^{2}&kq\\ \noalign{\medskip}0&{k}^{2}&0&-kq\\ \noalign{\medskip}0&0
&0&{k}^{2}\end {array} \right]
}
\newcommand{\Hslashglu}{
\left[ \begin {array}{cccc} 1&0&0&1\\ \noalign{\medskip}0&0&-1&0
\\ \noalign{\medskip}0&-1&0&0\\ \noalign{\medskip}0&0&0&1\end {array}
 \right] 
}

\newcommand{\Hising}{
\left[ \begin {array}{cccc} 1&0&0&1\\ \noalign{\medskip}0&1&1&0
\\ \noalign{\medskip}0&-1&1&0\\ \noalign{\medskip}-1&0&0&1\end {array}
 \right]
}
\newcommand{\Height}{
 \left[ \begin {array}{cccc} {p}^{2}+2\,pq-{q}^{2}&0&0&{p}^{2}-{q}^{2}
\\ \noalign{\medskip}0&{p}^{2}-{q}^{2}&{p}^{2}+{q}^{2}&0
\\ \noalign{\medskip}0&{p}^{2}+{q}^{2}&{p}^{2}-{q}^{2}&0
\\ \noalign{\medskip}{p}^{2}-{q}^{2}&0&0&{p}^{2}-2\,pq-{q}^{2}
\end {array} \right]
}

\newcommand{\Hf}{
 \left[ \begin {array}{cccc} {k}^{2}&0&0&0\\ \noalign{\medskip}0&0&kq
&0\\ \noalign{\medskip}0&kp&{k}^{2}-pq&0\\ \noalign{\medskip}0&0&0&{k}
^{2}\end {array} \right]
}
\newcommand{\Ha}{
\left[ \begin {array}{cccc} {k}^{2}&0&0&0
\\ \noalign{\medskip}0&0&kq&0\\ \noalign{\medskip}0&kp&{k}^{2}-pq&0
\\ \noalign{\medskip}0&0&0&-pq\end {array} \right]
}

\newcommand{\Rslash}{R_\backslash}
\newcommand{\Raslash}{R_\slash}

\beq       \label{eq:slash}       \hspace{-.2in}
\Rslash = \Hslash ,  \hspace{1cm}   \Raslash = \HslashDS, 
\eq  
\[ 
R_f = \Hf, \hspace{1cm} 
R_{fI} = \Hslashglu, \hspace{.21cm} 
\] 
\beq   \label{eq:fII}   
\R_{fII}= \Hslashglue, \hspace{.321cm} R_{fIII}=\Hslashgluex, 
\eq 

\[
R_a = \Ha, \hspace{1cm} R_{ag} = \Haglue, \hspace{.2in} 
\]
\[
\Hising, \;\; \Height 
\]

%% file: tex/N3.tex
\subsection{Braid representations in rank $N=3$}   \label{ss:repsN3}    $\;$  

In this section we address the extension of the classification of CC braid representations to the with-glue case. 
The strategy  is to use the theorem ~\ref{cth:HoG}
above and hence start with CC reps and extend from there, searching, in the first instance, using 
computer algebra to refine the general with-glue ansatz.

\subsubsection{The topography of a glue matrix}   $\;$  

A preliminary technical question is suitable names for the variables in the glue positions. 
In rank 2 there are five such, for now we can call them $p,q,r,s,t$ 
as for example in (\ref{eq:Rag}). 
Then the glue components corresponding to these by restriction in each of the three restrictions of 123, to 12, 13 and 23, can be 
$p_{12},q_{12},r_{12},s_{12},t_{12}$ and so on. 
We have (marking the CC positions simply with a 1):

\[  
\ignore{{
\begin{bmatrix} 
1 & \rg1 & \rg1 & \rg1 & \rg1 &\rg1 &\rg1 &\rg1 &\rg1 \\
  & 1 & \rg1 & 1 &\rg1&\rg1&\rg1&\rg1&\rg1 \\
  &   & 1 & &  \rg1& \rg1&1&\rg1&\rg1 \\
  & 1 & \rg1  &1&\rg1&\rg1&\rg1&\rg1&\rg1 \\
  &&&&1&\rg1&&\rg1&\rg1 \\
  &&&&&1&&1 &\rg1\\
  &&1&&\rg1&\rg1&1&\rg1&\rg1 \\
  &&&&&1&&1&\rg1 \\
  &&&&&&&&1
  \end{bmatrix}
\;\;  \leadsto \;\;
}}
  \begin{bmatrix} 
1 & \rg{p_{12}} & \rg{p_{13}} & \rg{q_{12}} & \rg{r_{12}} &\rgg{p_{1123}} &\rg{q_{13}} &\rgg{p_{1132}} &\rg{r_{13}} \\
  & 1 & \rgg{p_{1213}} & 1 &\rg{s_{12}}&\rgg{p_{1223}}&\rgg{p_{1231}}&\rgg{p_{1232}}&\rgg{p_{1233}} \\
  &   & 1 & &  \rgg{p_{1322}}& \rgg{p_{1323}}&1&\rgg{p_{1332}}&\rg{s_{13}} \\
  & 1 & \rgg{p_{2113}}  &1&\rg{t_{12}}&\rg{p_{2123}}&\rgg{p_{2131}}&\rgg{p_{2132}}&\rgg{p_{2133}} \\
  &&&&1&\rg{p_{23}}&&\rg{q_{23}}&\rg{r_{23}} \\
  &&&&&1&&1 &\rg{s_{23}}\\
  &&1&&\rgg{p_{3122}}&\rgg{p_{3123}}&1&\rgg{p_{3132}}&\rg{t_{13}} \\
  &&&&&1&&1&\rg{t_{23}} \\
  &&&&&&&&1
  \end{bmatrix}
\in \MATCHwg^3(2,2)
\]
The naming and colour coding here is  
black for CC; orange for glue between two charge sectors; red for glue between three. 
However, in the end we use simply $m_{i,j}$ for the glue positions in each ansatz below.

\subsubsection{Recalling the rank-3 CC classification}

In rank-3 we have, from \cite{MartinRowell2021}, the following classification of varieties of CC braid representations up to local equivalence. 
The first thing to recall is that rank-3 solutions are `built' from rank-2; and in rank-2 there are four solution types up to orientation, 
denoted /,a,f, and 0. The corresponding R-matrix types can be represented by:
\[
R_{\slash} = \mat a \\ &0&b \\ &c&0 \\ &&&d \tam , \hspace{.2cm} 
R_a = \mat a \\ &a+b&b \\ &-a&0 \\ &&&b \tam, 
\hspace{.2cm} 
R_f = \mat a \\ &a+b&b \\ &-a&0 \\ &&&a \tam, 
\hspace{.2cm} R_0 = Id_4 
\]
In rank-$N$ one codifies solutions using the complete graph $K_N$. 
One thinks of attaching a rank-2 solution to each edge. Thus in rank-3 we need three rank-2 solutions as input. 
And thus solutions are codified by triples from /,a,f,0. 
Some combinations are impossible. In summary we have: 
\\
///
\\
a//, \;  f//, \;  0//
\\
aa0, \; fff,\; , ff0, \;   f00, \;  aaf, \;  000\\

In \cite{MartinRowell2021} there are fairly detailed analyses of R-matrix spectra in all these cases. Here we will focus, for now, on a couple of interesting cases from our extension perspective.

\subsubsection{Braid representations extending the CC type aa0}

Recall that in the end we use simply $m_{i,j}$ for the glue positions. An example braid rep 
variety extending aa0 is, where we have used $m_i$ for the glue variables:

\ignore{{\[ \hspace{-5.4852cm} R=
\left[\begin{array}{ccccccccc}
a  & 0 & 0 & 0 & \frac{b m_{3,8}^{2}}{a^{2}}-\frac{m_{3,8}^{2}}{a} & -\frac{2 b m_{3,8}^{2} m_{4,7}}{a^{3}}+\frac{m_{3,8}^{2} m_{4,7}}{a^{2}}-\frac{m_{3,8} m_{4,8}}{a} & 0 & \frac{m_{3,8}^{2} m_{4,7}}{a^{2}}\!+\!\frac{m_{3,8} m_{4,8}}{a} & \frac{b m_{3,8}^{2} m_{4,7}^{2}}{a^{4}}-\frac{m_{3,8}^{2} m_{4,7}^{2}}{a^{3}} 
\\
 0 & a\! +\!b  & -\frac{b m_{4,7}}{a} & a  & -2 m_{3,8} & -\frac{b m_{3,8} m_{4,7}}{a^{2}}+\frac{m_{3,8} m_{4,7}}{a}-m_{4,8} & -m_{4,7} & \frac{3 m_{3,8} m_{4,7}}{a} & \frac{b m_{3,8} m_{4,7}^{2}}{a^{3}}-\frac{2 m_{3,8} m_{4,7}^{2}}{a^{2}}+\frac{m_{4,7} m_{4,8}}{a} 
\\
 0 & 0 & a  & 0 & 0 & -m_{3,8} & 0 & m_{3,8} & 0 
\\
 0 & -b  & \frac{b m_{4,7}}{a} & 0 & \frac{2 b m_{3,8}}{a} & -\frac{3 b m_{3,8} m_{4,7}}{a^{2}} & m_{4,7} & m_{4,8} & \frac{b m_{3,8} m_{4,7}^{2}}{a^{3}}-\frac{m_{4,7} m_{4,8}}{a} 
\\
 0 & 0 & 0 & 0 & b  & -\frac{2 b m_{4,7}}{a} & 0 & 2 m_{4,7} & \frac{b m_{4,7}^{2}}{a^{2}}-\frac{m_{4,7}^{2}}{a} 
\\
 0 & 0 & 0 & 0 & 0 & 0 & 0 & a  & 0 
\\
 0 & 0 & 0 & 0 & 0 & -\frac{b m_{3,8}}{a} & a  & \frac{b m_{3,8}}{a} & 0 
\\
 0 & 0 & 0 & 0 & 0 & -b  & 0 & a +b  & 0 
\\
 0 & 0 & 0 & 0 & 0 & 0 & 0 & 0 & a  
\end{array}\right]
\]
}}

\[\hspace{-1.4852cm}
\left[
\footnotesize{\begin{array}{ccccccccc}
a & 0 & 0 & 0 &
\dfrac{b m_1^2}{a^2}-\dfrac{m_1^2}{a} &
-\dfrac{2 b m_1^2 m_2}{a^3}+\dfrac{m_1^2 m_2}{a^2}-\dfrac{m_1 m_3}{a} &
0 &
\dfrac{m_1^2 m_2}{a^2}+\dfrac{m_1 m_3}{a} &
\dfrac{b m_1^2 m_2^2}{a^4}-\dfrac{m_1^2 m_2^2}{a^3}
\\
0 & a+b & -\dfrac{b m_2}{a} & a &
-2 m_1 &
-\dfrac{b m_1 m_2}{a^2}+\dfrac{m_1 m_2}{a}-m_3 &
- m_2 &
\dfrac{3 m_1 m_2}{a} &
\dfrac{b m_1 m_2^2}{a^3}-\dfrac{2 m_1 m_2^2}{a^2}+\dfrac{m_2 m_3}{a}
\\
0 & 0 & a & 0 & 0 & -m_1 & 0 & m_1 & 0
\\
0 & -b & \dfrac{b m_2}{a} & 0 &
\dfrac{2 b m_1}{a} &
-\dfrac{3 b m_1 m_2}{a^2} &
m_2 & m_3 &
\dfrac{b m_1 m_2^2}{a^3}-\dfrac{m_2 m_3}{a}
\\
0 & 0 & 0 & 0 &
b &
-\dfrac{2 b m_2}{a} &
0 & 2 m_2 &
\dfrac{b m_2^2}{a^2}-\dfrac{m_2^2}{a}
\\
0 & 0 & 0 & 0 & 0 & 0 & 0 & a & 0
\\
0 & 0 & 0 & 0 & 0 &
-\dfrac{b m_1}{a} &
a & \dfrac{b m_1}{a} & 0
\\
0 & 0 & 0 & 0 & 0 & -b & 0 & a+b & 0
\\
0 & 0 & 0 & 0 & 0 & 0 & 0 & 0 & a
\end{array}}
\right]
\]

This gives  a multidimensional variety of extended solutions over every representative aa0 solution.  However, generically the presence of glue here can be gauged away by means of a local basis change--the result is the charge conserving shadow obtained by tuning all glue to $0$.

Another example, with $a=b$, is \[\left[ \begin {array}{ccccccccc} b&0&0&0&0&-m_{{1}}&0&m_{{1}}&0
\\ \noalign{\medskip}0&2\,b&-m_{{3}}&b&-2\,m_{{2}}&-m_{{4}}&-m_{{3}}&3
\,{\frac {m_{{3}}m_{{2}}}{b}}&-m_{{5}}\\ 0&0&b&0&0&-
m_{{2}}&0&m_{{2}}&0\\ 0&-b&m_{{3}}&0&2\,m_{{2}}&-3\,
{\frac {m_{{3}}m_{{2}}}{b}}&m_{{3}}&m_{{4}}&m_{{5}}
\\ 0&0&0&0&b&-2\,m_{{3}}&0&2\,m_{{3}}&0
\\ 0&0&0&0&0&0&0&b&0\\ 0&0&0&0&0&-
m_{{2}}&b&m_{{2}}&0\\ 0&0&0&0&0&-b&0&2\,b&0
\\ 0&0&0&0&0&0&0&0&b\end {array} \right], 
\] which has exactly one eigenvalue.  This variety of solutions has points that cannot be locally gauged to a charge conserving matrix, an example of which is:
\[ \left[ \begin {array}{ccccccccc} 1&0&0&0&0&0&0&0&0
\\ 0&2&0&1&0&0&0&0&-1\\ 0&0&1&0&0&0
&0&0&0\\ 0&-1&0&0&0&0&0&0&1\\ 0&0&0
&0&1&0&0&0&0\\ 0&0&0&0&0&0&0&1&0
\\ 0&0&0&0&0&0&1&0&0\\ 0&0&0&0&0&-
1&0&2&0\\ 0&0&0&0&0&0&0&0&1\end {array} \right]. 
\]
Note that this matrix is unipotent.

\mdef 
We can use the solution above to address the rank-restriction question. 
The 12 part of the representation is 
\[
\left[\begin{array}{ccccccccc}
a  & 0 & 0  & \frac{b m_{3,8}^{2}}{a^{2}}-\frac{m_{3,8}^{2}}{a} 
\\
 0 & a\! +\!b  & a  & -2 m_{3,8} 
\\
 0 & -b  &  0 & \frac{2 b m_{3,8}}{a} 
 \\
 0&0&0&b
\end{array}\right]
\]
which is indeed a braid representation. The 13 part is
\[
\left[\begin{array}{ccccccccc}
a  & 0 & 0  & m 
\\
 0 & a  & 0  & 0 
\\
 0 & 0  &  a & 0 
 \\
 0&0&0&a
\end{array}\right]
\]
which is not a representation. 
From this we see that the very powerful rank-restriction property of CC representations does not extend to glue cases. 

\subsubsection{Braid representations extending the CC type fff, unipotent case} Here we begin with the fff, unipotent solution.
After applying gauge transformations, we arrive at the remarkably simple case, which is exactly the unipotent fff case with a single piece of glue: the $1$ is the (5,9) position.

\[ \left[ \begin {array}{ccccccccc} 1&0&0&0&0&0&0&0&0
\\ 0&2&0&-1&0&0&0&0&0\\ 0&0&2&0&0&0
&-1&0&0\\ 0&1&0&0&0&0&0&0&0\\ 0&0&0
&0&1&0&0&0&1\\ 0&0&0&0&0&2&0&-1&0
\\ 0&0&1&0&0&0&0&0&0\\ 0&0&0&0&0&1
&0&0&0\\ 0&0&0&0&0&0&0&0&1\end {array} \right] 
.\]
This case is perfect for a quick representation-theory analysis, and we report on this in \S\ref{ss:unipotentRT}. 

 \subsubsection{Some $a//$ cases}
 Starting with the generic $a//$ form and adding glue we find that there are interesting solutions.  For concreteness we take the following CC solution 
 \[R:=\left[ \begin {array}{ccccccccc} 1&0&0&0&0&0&0&0&0
\\ 0&2&0&1&0&0&0&0&0\\ 0&0&0&0&0&0
&1&0&0\\ 0&-1&0&0&0&0&0&0&0\\ 0&0&0
&0&1&0&0&0&0\\ 0&0&0&0&0&0&0&1&0
\\ 0&0&1&0&0&0&0&0&0\\ 0&0&0&0&0&1
&0&0&0\\ 0&0&0&0&0&0&0&0&1\end {array} \right].
\]  Note that the Jordan form of $R$ has a single $2\times 2$ block with eigenvalue $1$ and 7 $1\times 1$ blocks, two of which have eigenvalue $-1$ and the remaining $5$ have eigenvalue $1$.  Adding glue and solving we find, for example the following solution: \[ R_g:=\left[ \begin {array}{ccccccccc} 1&0&0&0&4&-2&0&-2&1
\\ 0&2&-1&1&-2&1&0&-2&1\\ 0&0&0&0&0
&0&1&0&0\\ 0&-1&1&0&2&-1&0&2&-1\\ 0
&0&0&0&1&0&0&0&0\\ 0&0&0&0&0&0&0&1&0
\\ 0&0&1&0&0&0&0&0&0\\ 0&0&0&0&0&1
&0&0&0\\ 0&0&0&0&0&0&0&0&1\end {array} \right] 
.
\] 

One computes that the Jordan form has 2 $2\times 2$ blocks with eigenvalue $1$.  Moreover, there is no local gauge choice transforming $R_g$ to a charge conserving matrix.

%% file: tex/N2-reps0.tex
\section{Applications II: representation theory}

\subsection{Braid representation theory generalities}

\newcommand{\id}{Id}

Fix $N$. Let $R$ be a rank-$N$ $R$-matrix, hence giving a braid rep $F_R:\Bcat \rightarrow \MAT^N$ by 
$F_R(\sigma) = R$. 
Indeed if $R$ is denoted $R_\chi$, say, then we may write just $F_\chi$ for the braid rep $F_{R_\chi}$; 
and $\rho^\chi_n$ for the ordinary representation of $B_n$ this gives. 

If an R-matrix $R$ is fixed then we write 
\beq  \label{eq:defRi}
R_i = \id_N^{\otimes i-1} \otimes R \otimes  \id_N^{\otimes n-i-1}
\eq  

We write 
$$
\AAAA^{\chi}_{n}   = \langle  R_1 , R_2, ..., R_{n-1}  \rangle  \;\; \subset \MAT^N(n,n) 
$$ 
for the image algebra of the rep of $B_n$ given by  $R_\chi$.  
We write ${\AAAA'}^{\chi}_{n}$ for the commutant. 
\medskip

\newcommand{\J}{{\mathsf{J}}}

\mdef  Fix $N\in\N$. 
For $j \in \{1,2,..,N\}$ define $\overline{j} = N+1-j$. 
For a 
word $w = w_1 w_2 \cdots w_l \in \{1,2,\ldots,N\}^l$ define $\overline{w} = \overline{w_1} \overline{w_2} \cdots \overline{w_l}$.

Fix $N$. 
The matrix $\J \in \MAT^N(1,1)$ is given by 
$\J_{i,j} = \delta_{i,N+1-j} = \delta_{i,\overline{j}}$. 

E.g., for $N=2$ we have $\J_{} =\begin{bmatrix}
    0 & 1\\ 1&0
\end{bmatrix}$.   

Where convenient we may 
write $\J_n$ for the $\J$ with $N=n$. 
Thus for example  
$
\J^{\otimes n} =  \J_N^{\otimes n} = \J_{N^n}  .
$

Note that conjugation by $\J^{\otimes n}$ in $\MAT^N(n,n)$ transforms a matrix $R$ by 
$R_{v,w} \mapsto R_{\overline{v},\overline{w}}$. 

For example, with $ N=3$ we have $(\J^{\otimes 2} M \J^{\otimes 2})_{13,23} = M_{31,21}$. 

In other words if $M \mapsto M^T$ is transpose and, for a square matrix, 
$M \mapsto M^\dashv $  is skew-transpose,
then $\J M \J = (M^T)^\dashv$.   

\medskip 

The following is {possibly} well-known, at least in rank $N=2$ (see e.g. \cite{Hietarinta92}):

\begin{lemma}  \label{lem:skewif}
Let $N \in \N$. 
    If a YBO $R\in Aut(\C^N\otimes \C^N) \subset \MAT^N(2,2)$ is invariant under the skew transpose 
    $R_{ij,kl} \mapsto R_{\overline{kl},\overline{ij}}$ 
    then for each $n \in \N$ the representation $\rho_R$ given by 
    the functor acting on $B_n$, i.e by $F_R(B_n)$, is equivalent to its dual 
    $\rho_{R^T}   = \rho_{R}^T$. 
    \ppmm{   
    as given by $ (\rho^R_n)^T$, which can also be written $\rho^{R^T}_n$.}
\end{lemma}
\begin{proof}  
\ignore{Define $J_N$ to be the $N\times N$ skew diagonal matrix with non-zero entries all $1$, eg., for $N=2$ we have $J_2=\begin{bmatrix}
    0 & 1\\ 1&0
\end{bmatrix},$ and $J_{N^n}:=J_k^{\otimes n}$.}  
    Note that conjugation by $J_{N^2}$ carries any matrix to the simultaneous transpose and skew transpose.  Thus conjugation of $R_i$ by $J_{N^n}$ yields $(R^T)_i$ since $R$ is invariant under skew-transpose.  Thus $J_{N^n}$ intertwines the representations $\rho_R$ and $\rho_{R^T}$.
\end{proof}

%% file: tex/N2-reps.tex
\subsection{Representation theory in rank $N=2$}   \label{ss:repsN2}

In preparation for passing to the rank $N=3$ case it will be useful briefly to consider 
some of these rank-2 representations from the algebraic perspective - the effect of added glue on representations per se. 

\subsubsection{Little-rank $n=2$}

One of the simplest aspects to consider is the equivalence class of representation of $B_2$ that we get. 
This is largely determined in effect by the Jordan form of the image of the elementary braid, the $R$-matrix itself. 

Another way to think of this is in terms of the `local relation' - the 
minimal  
polynomial of $R$. 

\mdef 
First consider a-glue, i.e. $R_{ag}$. 
If $p \neq -q$ it will be clear that the Jordan form is diagonal irrespective of the value of $k$. 
On the other hand, if $p=-q$ then while the central block becomes a single Jordan block (note that it is unipotent but not the identity matrix), the `outer' block now also becomes a single Jordan block unless $k=0$. 

The local relation is always $(R-p)(R+q)=0$ here - quadratic, so that our braid representation is always also a Hecke representation. This holds even if $p=-q$ and the quadratic is degenerate. 
In Hecke parameterisation one often normalises so that eigenvalue $p=1$ and eigenvalue $-q = -v^2$ 
(indeed the Hecke parameter here called $v$ is usually called $q$, but we already have an ambiguity to avoid here, hence $v$). The case $p=-q$ thus corresponds to $1=-v^2$, so $v=i$ - the `root of unity'. 
Thus $[2] := v+ v^{-1} = i-i=0$. 

Note that this ag case works as a proof that glue in our present sense changes the representation theory, 
by (sometimes) adding more radical. 
However, this example is too small to investigate every aspect. This is a representation of braid that is also a representation of Hecke, and the dimension of the image does not go up, because the whole radical is already present in the inner block. 

\mdef 
Let us turn  
now to the various f cases. 
\medskip 
\\ 
For f itself the local relation is again quadratic, so we always have a Hecke representation (indeed Temperley--Lieb). 
This classical case is completely understood in all ranks for all parameter values. 
\medskip 
\\ 
For f-glue-I however, we see that the local relation is 
\beq  \label{eq:cubicz} 
(R-1)^2(R+1)=0
\eq   
So this is not Hecke.
By the HoG machinery it has a quotient (only quotienting by a subideal of the radical) that is Temperley--Lieb, 
so, for example, its simple part is Temperley--Lieb. But overall it is not even Hecke. 
We address the ordinary  representation theory in all little-ranks $n$ in \S\ref{ss:caseRfI}.    \medskip 
\\
For f-glue-II, as in (\ref{eq:fII}), we have the following points. 
Firstly we take $k=1$ without significant loss of generality. 
Thus the CC part is X-equivalent to the CC part of $R_{fI}$ and hence $\infty$-equivalent by \cite[Corollary 7.24]{MRT25}; so, by Theorem \ref{cth:HoG}, the simple representation theories of $R_{fI}$ and $R_{fII}$ coincide--they are encoded in the CC parts.  Remarkably, the specialisation  of $R_{fII}$ to $p=q=0$ and $s=1$ (called $\R_\diamondsuit$ below) is  $\infty$-equivalent to $R_{fI}$ on the nose, as we will prove in Proposition \ref{prop: fI and diamond equiv}. 

Interesting points on the variety to consider are 
\[
R_{\diamondsuit} = \left[ \begin{array}{cccc} 
1&0&0&1   \\ \noalign{\medskip}
0&0&1&0  \\ \noalign{\medskip}
0&1&0&0  \\ \noalign{\medskip}
0&0&0&1  \end{array}
 \right] 
 , \hspace{.51cm} 
 R_{\spadesuit} = \left[ \begin{array}{cccc} 
1&1&1&1   \\ \noalign{\medskip}
0&0&1&1  \\ \noalign{\medskip}
0&1&0&1  \\ \noalign{\medskip}
0&0&0&1  \end{array}
 \right] 
 , \hspace{.51cm} 
 R_{\heartsuit} = \left[ \begin{array}{cccc} 
1&1&0&1   \\ \noalign{\medskip}
0&0&1&1  \\ \noalign{\medskip}
0&1&0&0  \\ \noalign{\medskip}
0&0&0&1  \end{array}
 \right] 
\]

The cases 
$R_{\diamondsuit}$ and 
$R_{\spadesuit}$  
have 
minimal polynomial $(R-1)^2(R+1)$, like $R_{fI}$; 
while $R_{\heartsuit}$  
has minimal polynomial $(R-1)^3(R+1)$ (and hence is not equivalent to $R_{fI}$, cf. (\ref{eq:cubicz})). 

\ignore{{
Aside: It is interesting to note that $R_{\diamondsuit}$ indeed has the same structure as
$R_{fI}$, computed in similar time.  Meanwhile $R_{\spadesuit}$ takes far longer (indeed so far does not complete $n=6$. (And $R_{\heartsuit}$ not yet attempted.)   ... 
... 
}}

It is convenient to summarize some low rank cases with a table as in Table~\ref{tab:magm}.

\begin{table}[]
\centering
\begin{tabular}{c|cccc|ccc|ccc}
         & dim         & dim          & dim     & dim    & dim & dim & dim &dim&dim & dim \\ 
        n & algebra $\AAAA^{fI}_n$ & $\AAAA^{fI}_n$/$rad$ & $rad^2$ & $rad^3$ & $\AAAA^{\spadesuit}_n$ & $\AAAA^{\spadesuit}_n / rad$  & $rad^2$ & $\AAAA^{\heartsuit}_n$ & /$rad$ & $rad^2$ \\ \hline 
        2 &      3        &   2            & 0   & 0  & 3 & 2 & 0 & 4 & 2&1 \\
        3 &     10       &  5           &  0    & 0  & 10 & 5 & 0 & 20 & 5 &6 \\
        4 &     35      & 14          & 1     & 0   &  35 & 14 & 1 & 70 & 14 &28 \\
        5 & 126        &  42         &  9     & 0   & 126 & 42 & 9 & 252 & 42&120 \\ 
        6 & 462       &  132       & 55      & 1
\end{tabular}
\caption{
Dimensions of parts of the image algebra $\AAAA^{fI}_n$ for the rep $fI$, 
and $\AAAA^{\spadesuit}_n$ 
and $\AAAA^{\heartsuit}_n$ 
(partly according to a Magma calculation \cite{Magma}) 
in low little-rank $n$. 
Notice the semisimple quotient dimensions are Catalan as required.}
    \label{tab:magm}
\end{table}

\ignore{{
\subsubsection{Little-rank $n=3$}

$\;$ 

\mdef Next we can consider little-rank $n=3$.  
For $R$ we have ...   ...
}}

\subsubsection{A case study: the case $R_{fI}$}   \label{ss:caseRfI}

First we will explain why $R_{fI}$ is an illuminating case to study; and then we will analyse the representation theory in all little-ranks $n$. 
The point is that the HoG Theorem immediately gives the simple content of the representations for all $n$ here (and the de-glued version is semisimple). 
So we can demonstrate the effect of the glue by looking at the non-semisimple structure. 

Consider the $R_{fI}$ solution:  $\left[ \begin {array}{cccc} 1&0&0&1\\ \noalign{\medskip}0&0&-1&0
\\ \noalign{\medskip}0&-1&0&0\\ \noalign{\medskip}0&0&0&1\end {array}
 \right] .$

\mdef 
Observe that the de-glued version of this is the classical rank-2 tensor space representation of the symmetric group, Schur--Weyl dual to the usual $sl_2$ action. 
So the de-glued version is semisimple, with dimensions of simple summands given by the usual truncated Pascal triangle; and multiplicities given by the corresponding dimensions of simple $sl_2$-modules. 
Indeed we know the irreducible content of each CC block. 
The CC blocks $Y_c$ are labelled by compositions $c$ of $n$ into two parts - each of which compositions corresponds to an integer partition $\overline{c} = \lambda = (n-m,m)$ by reordering if necessary. 
The isomorphism class of $Y_c$ depends only on $\lambda$, and 
\[
Y_\lambda \cong \bigoplus_{\lambda' \leq \lambda} S_{\lambda'} 
\]
where $\lambda' \leq \lambda$ if $\lambda' \in \{ (n-m,m), (n-m+1,m-1), ... \}$. 
For example $Y_{(3,1)} = S_{(3,1)} \oplus S_{(4)}$ (as expected, since $Y_{(3,1)}$ coincides with the usual defining representation of $\Sym_4$). 

By the HoG Theorem \ref{cth:HoG} then the rep given by $R_{fI}$ has exactly the same simple content (simple filtration factors) as the de-glued version above, for all $n$. 
So next we turn to the structure with the glue in place - which we know from \ref{cth:HoG} can only have the effect of making the sum of simple modules non-split. 

So next we will determine the size of the indecomposable blocks for each $n$. 

\mdef   \label{lem:already}
Note that we already know that there are at least two indecomposable blocks in each case, by \ref{lem:parity}. 
That is, different parity sectors are necessarily in different blocks, since $R_{fI}$ itself preserves parity. 

Note also that the two different parity sectors are easily seen to be of the same size. 

\medskip

 Define $\mathscr{A}_n=\langle R_1,\ldots, R_{n-1}\rangle$ to be the image of the braid group algebra under the representation $\rho_n = \rho_{R_{fI},n}$. 
 Define $\AAAA'_n$ to be the commutant of this tensor space algebra. 

\mdef \label{prop:a-glue decomp} \textbf{Proposition}
For $n\geq 3$ the representation $\rho_n$ 
obtained from $R_{fI}$ 
decomposes as the direct sum of two indecomposable representations of dimension $2^{n-1}$.
\medskip

\noindent 
{\em{Proof}}.  Define $D_1=\begin{bmatrix}a&0\\0&b\end{bmatrix}$. For any matrix $M$ with entries in the function field $\C(a,b)$ 
here 
define $M'$ to be the image of $M$ under the automorphism defined by $a\leftrightarrow b$. Thus, for example, $D_1'=\begin{bmatrix} b&0\\0&a\end{bmatrix}$.  For $n\geq 2$ define $D_n=\begin{bmatrix} D_{n-1}&0\\0&D_{n-1}'\end{bmatrix}$.
For $n\geq 3$ we 
\ppmm{will now} 
use induction to prove that:\begin{enumerate}
    \item if $T\in\AAAA_n'$ then there exists $N_1$ and $N_2$ strictly upper triangular and $B\in \AAAA'_{n-1}$ such that  \[T=\begin{bmatrix} D_{n-1}+N_1 & B\\ 0& D'_{n-1}+N_2 \end{bmatrix}.\]
    \item \(\begin{bmatrix} D_{n-1} & 0\\ 0&D'_{n-1} \end{bmatrix}\in \AAAA'_n.\)
\end{enumerate}
Direct calculation shows that any $T\in\AAAA'_3$ has the form:  \[\left[ \begin {array}{cccccccc} a&h&-h&c&h&-c&c&d
\\ \noalign{\medskip}0&b&0&f&0&-f&e&g\\ \noalign{\medskip}0&0&b&-f&0&e
&-f&-g\\ \noalign{\medskip}0&0&0&a&0&0&0&h\\ \noalign{\medskip}0&0&0&e
&b&-f&f&g\\ \noalign{\medskip}0&0&0&0&0&a&0&-h\\ \noalign{\medskip}0&0
&0&0&0&0&a&h\\ \noalign{\medskip}0&0&0&0&0&0&0&b\end {array} \right] \]

so (both) statements hold for $n=3$.  

Suppose the statement holds for $\AAAA_{n-1}$, and consider $\AAAA_{n}$.
 For (1): if $T$ commutes with $I\otimes \AAAA_{n-1}\subset\AAAA_n$ then $T$ has a $2\times 2$ block structure with each block of the given form, by hypothesis.  Then commutation with $R_1$ yields the desired form.  

For (2): the given diagonal matrix commutes with $I\otimes \AAAA_{n-1}$ by hypothesis, and by $R_1$ by direct verification of the block $4\times 4$ matrices.

From (1) we conclude that the representation $\rho_n$ decomposes into \emph{at most} two indecomposable summands, since any matrix in $\AAAA'_n$ has at most 2 distinct eigenvalues. From (2) we 
\ppmm{confirm (cf. \ref{lem:already})}
it decomposes into \emph{at least} 2 indecomposable summands, each of dimension $2^{n-1}$ by observing that the eigenvalues $a,b$ each appear with multiplicity $2^{n-1}$. 
\qed
\medskip  

\mdef 
Observe that this gives a lower bound on the dimension of the radical which shows 
that it is rapidly growing with $n$ 
(the dimensionally-cheapest way to make the sum in a block indecomposable is to glue - in the ordinary radical sense - every summand to one of the 1-dimensional summands). 
Since the de-glued case has no radical for any $n$, this is a nice example of the 
profound effect of adding glue. 

\vspace{.2cm}

Next we give an interesting example of an \ppmm{algorithmic entanglement} - a natural isomorphism of braid representations (in this case between $R_{fI}$ and  $R_{\diamondsuit}$)   where the transformation is not built by local matrices 
(if there is no such local transformation we say the reps are not locally or `gauge' equivalent).

Aside: One may easily 
verify  
that $R_{fI}$ and $R_{\diamondsuit}$  are not gauge equivalent.

\mdef\label{prop: fI and diamond equiv}  {\bf{Proposition}}.
The braid representations given by $R_{fI}$ and $R_{\diamondsuit}$ 
(giving the ordinary representation sequences 
$\rho^{fI}$ and $\rho^\diamondsuit$) are $\infty$-equivalent, by means of a diagonal intertwiner \ppmm{in each little-rank $n$}.  


\medskip 

\newcommand{\X}{\mathcal{X}}

\noindent {\em Proof}.     
We will   
define a specific candidate for the set of intertwiners, then show that it works, using 
properties of its construction that we will first exhibit.

Define $4$ diagonal matrices 
$A=diag(1,-1,1,1), \; 
 B=diag(1,1,-1,1), \;$ $ C=diag(-1,-1,1,-1)$ and $D=diag(-1,1,-1,-1)$.  
\ppmm{Observe that}
for $X\in \mathcal{X}=\{A,B,C,D\}$ we have 
$R_{fI} \; X \;=\; X \; R_\diamondsuit$.  
For $X_i\in\mathcal{X}$ denote the direct sum $X_1\oplus \cdots \oplus X_n$ by the word $X_1X_2\cdots X_n$.

Next define a map on the words in $\mathcal{X}$ as follows. On generators set:
$\varphi(A)=BA, \; \; \varphi(B)=AC, \;\; \varphi(C)=DB$ and $\varphi(D)=CD$, and then extend multiplicatively.  
Next define words inductively by $w_1=A$ and $w_{n}=\varphi(w_{n-1})$.  
Thus, for example $w_2=BA$ and $w_3=ACBA$.   
\ppmm{
Let $T_n$ be the diagonal matrix of size $2^{n}$ corresponding to the word $w_{n-1}$.  }

We claim that 
$\rho_n^{fI}(\beta \; )T_n \;=\; T_n \; \rho_n^\diamondsuit(\beta)T_n$, 
for all $\beta\in B_n$ i.e. $T_n$ intertwines $\rho_n^{fI}$ and $\rho_n^\diamondsuit$.
It is enough to verify on generators $R_i$. 
In little-rank $n=2$ we have the intertwiner property for the R-matrices as already observed. 
To do all ranks it will be convenient to set up some machinery as follows. 

\medskip 

Write $g$ for the map from words in $\X$ to words in $\{\pm1\}$ by the obvious substitution, 
$g(A)=1x11$ (we write $x$ for -1 for ease of parsing as part of a string)   and so on. 
Of course $g$ is injective and $g(wv)=g(w)g(v)$. 
The formal inverse $g'$ of $g$ is `well-defined' when a word in $\{\pm1\}$   lies in the image of $g$. 

Let us write $g_o(w)$ and $g_e(w)$ for the odd and even position subsequences of $g(w)$, respectively. 
And $g_-(w)=g_o(w) g_e(w)$. 
For example since $g(BA)=11x11x11$ 
we have $g_e(BA)=11x1$; and $g_-(BA) = 1x11 11x1 $, 
which has 
$g'(g_-(BA)) = AB$. 
Indeed $g'(g_-(A))=B$ and $g'(g_-(B))=A$ and the previous identity follows.

We say that a word $w$ in $\X$ is `well-combed' if it is length 1 or if 
$g'(g_e(w))$ and $g'(g_o(w))$  
are well-defined (hence $g'(g_-(w))$ also); and  
$g'(g_-(g'(g_-(w))))$ is also well-defined, and so on - iterating until the parts $g'(g_e(g'(g_e ... (w))))$ and so on are length 1.

\medskip 

\newcommand{\leadstog}{\stackrel{g_-}{\leadsto}}

The full sequence at $n=3$ is  $g(BA)=$ 11x11x11. 
Since $R_1 =  R\otimes 1_2 =R \oplus R$ the candidate intertwiner works for $R_1$. 
For $R_2 = 1_2 \otimes R$ its R-blocks meet positions 1357 and 2468, which are 
$g_o(BA)=$ 1x11 $\sim A$ and $g_e(BA)=$ 11x1 $\sim B$, which are again blockwise intertwiners. 
We can summarize as
\[
BA \leadstog AB 
\]

For $n=4$ we have ACBA giving  
 1x11xx1x11x11x11. 
Again there is nothing further to check for $R_1 = R\otimes 1_4 = R \oplus R \oplus R \oplus R$. 
For $R_2  = 1_2 \otimes R \otimes 1_2 = 1_2 \otimes (R\oplus  R)$ we extract positions 13579135 (dropping the second digit for alignment)  
giving $g_o(ACBA)$= 11x11x11 which is valid - it is $g'(g_o(ACBA))=BA$; 
and 24680246 giving  x1xx11x1, which is valid - it is $g'(g_e(ACBA))=CB$. 
For $R_3 = 1_4 \otimes R$ we extract positions 1593 (from the original 16 digit word) giving 1x11 $\sim A$; 
2604 giving xx1x $\sim C$; 3715 giving 11x1 $\sim B$; and 4826 giving 1x11 $\sim A$,
all of which are valid. 
\\
Observe that since $g_o$ extracts odd position entries, 
the odd position entries from $g_o(ACBA)$ give the 1593 entries from the original: $g_o(g_o(ACBA))= g_o(11x11x11) = 1x11 $ (abusing notation slightly; 
and  the even positions give the 3715 entries: $ g_e(g_o(ACBA))= g_e(11x11x11) = 11x1$.
Meanwhile $g_o(g_e(ACBA))  $ takes odd positions from the even position subset, so positions 2604 from the original:  $g_o(g_e(ACBA)) =g_o(x1xx 11x1) = xx1x $; 
and finally $g_e(g_e(ACBA)) = g_e(x1xx 11x1) = 1x11$ gets positions 4826. 
\\ 
We can write in summary 
\[
ACBA \leadstog BACB \leadstog ACBA 
\]

For $n=5$ we have BADBACBA giving 
 11x1 1x11 x1xx 11x1 1x11 xx1x 11x1 1x11
(separated just to aid counting). 
The decomposition into odd and even parts for $R_2$ gives: 1x11 xx1x 11x1 1x11 and 11x1 1x11 x1xx 11x1, which is valid. 
(These recode as ACBA and BADB)
The decomposition mod 4, for $R_3$, can be done by iterating the odd/even, to give 
11x1 1x11 $\sim BA$, x1xx 11x1 $\sim DB$, 
1x11 xx1x  $\sim AC$, 11x1 1x11 $\sim BA$. 
Finally mod 8, for $R_4$, we have 1x11, 11x1, xx1x, 1x11, 11x1, x1xx, 1x11, 11x1. 
Observe the summary:
\[
BADB \; ACBA \leadstog ACBA \; BADB \leadstog BADB \; ACBA \leadstog ABCA \; BDAB
\]

Overall we see that our claimed intertwiners work provided that all the necessary applications of $g_-$ are well-defined (each giving the verification of intertwining for a corresponding $R_i$. 
It remains only to show that every application of $g_-$ is well-defined - i.e. produces a word in the image of $g$. 

Observe that the first (4-term) factor in $g_-(X_1 X_2 X_3 ...)$ comes from positions 1357 of the original 
(recall $g_o$ picks factors in positions 135791357...), so depends only on $X_1 X_2$. The next factor (for sufficiently long word) comes from positions 9135, so  only depends on $X_3 X_4$. And so on until:  
The first factor in the $g_e$ part comes from positions 2468, so only depends on $X_1 X_2$. 
And the next depends on $X_3 X_4$, and so on. 
Thus to show that every application of $g_-$ is well-defined we only need to show that the image of every 
word of length 2 (that can arise) is well-defined. 
Considering $g_-(BB) = g_o(11x1 11x1) g_e( 11x1 11x1 ) = 1x1x ...$, which is not well-defined, 
we must establish that BB and some other sequences do not arise as $X_{2i-1} X_{2i}$. 
But this is clear from $\varphi$ - every such pair arises as the image of a single symbol, and so the total 
set of possibilities is BA, AC, DB, CD. 
So it only remains to check these. We already checked BA, AC. For DB we have x1xx 11x1, giving xx1x and 1x11.
For CD we have xx1x x1xx giving x1xx and xx1x.
\qed 


\subsubsection{A comparative case study: the point $R_{\diamondsuit}$ in $R_{fII}$}    \label{ss:casediamond}

$\;$ 

\newcommand{\II}{{\mathfrak{I}}}  

It is easy to verify here that every matrix of the form 
\beq \label{eq:commfII} 
\left[\begin{array}{cccccccc}
a  & h  & h  & c  & h  & c  & c  & d  
\\
 0 & b  & 0 & f  & 0 & f  & e  & g  
\\
 0 & 0 & b  & f  & 0 & e  & f  & g  
\\
 0 & 0 & 0 & a  & 0 & 0 & 0 & h  
\\
 0 & 0 & 0 & e  & b  & f  & f  & g  
\\
 0 & 0 & 0 & 0 & 0 & a  & 0 & h  
\\
 0 & 0 & 0 & 0 & 0 & 0 & a  & h  
\\
 0 & 0 & 0 & 0 & 0 & 0 & 0 & b  
\end{array}\right]
\eq  
lies in the commutant. 
(Probably it is instructive to compare with the $fI$ commutant in Proposision \ref{prop:a-glue decomp} above.) 
In fact we claim every element of the commutant takes this form. This is a stronger claim and so harder to verify. But let us take it on trust for now. 
Then it is clear that $\AAAA'$ has a 2d semisimple part ($a,b$); and a 6d radical ($c,d,e,f,g,h$). 
This seems to suggest two non-isomorphic indecomposable summands in $\AAAA$. 

By parity 
(as in (\ref{lem:parity}))
the two blocks are the 111,122,212,221 block; and the 112,121,211,222 block. 
\\
Both blocks have simple content 1/2/1 (writing a simple module just as its dimension) in some order. 
Let us try the following ansatz for the socle series for one of them: 
\[
\II_{odd}' \; = \; \begin{array}{c}
    1   \\
     2 \\
     1
\end{array}
\]
- note this allows a class potentially bigger than an isomorphism class. 
This would allow two self-maps: the self-isomorphism and the head to socle map. 
However, since each indecomposable/block module is the contravariant dual of the other 
(via the transpose isomorphism from \ref{lem:skewif}), and this structure is self-dual, this would imply two isomorphic modules, and hence a 2x2 matrix algebra in the commutant,
contradicting 
the completeness of 
(\ref{eq:commfII}). So we must reject this ansatz. 

Instead we can try 
\[
\II_{odd} = \; \begin{array}{c}
    1  \oplus 2 \\
     1
\end{array} ,   \hspace{1cm} 
\II_{eve} = \; \begin{array}{c}
 1 \\
    1  \oplus 2 
\end{array}
\]
Again each has two self-maps - isomorphism; and a map taking the 1 in the head to the 1 in the socle. And there is no isomorphism between them. Instead there are head-to-socle 1-maps both ways; and a head-to-socle 2-map from odd to even. 
%
This agrees in part with what we see in \ref{eq:commfII}, where the $a$ sector has $a$ and $c$;
and the $b$ sector has $b$ and $g$. 

\subsubsection{The case $R_{\heartsuit}$}

We can also determine the indecomposable blocks in all ranks in this case. 
We proceed as follows

Direct computation shows that  $T\in(\AAAA_3^\heartsuit)'$ has the form: 
\[\left[ \begin {array}{cccccccc} a&b&b&c&b&c&c&d\\ \noalign{\medskip}0
&a&0&b&0&b&0&c\\ \noalign{\medskip}0&0&a&b&0&0&b&c
\\ \noalign{\medskip}0&0&0&a&0&0&0&b\\ \noalign{\medskip}0&0&0&0&a&b&b
&c\\ \noalign{\medskip}0&0&0&0&0&a&0&b\\ \noalign{\medskip}0&0&0&0&0&0
&a&b\\ \noalign{\medskip}0&0&0&0&0&0&0&a\end {array} \right].\] 
This provides the base case for the proof by induction of the following:

\begin{lemma}\label{heartlem}
    Any $T\in(\AAAA_n^\heartsuit)'$ for $n\geq 3$ has the form: \[\begin{bmatrix} aI+N & B\\ 0& aI+N \end{bmatrix}\] where $N$ is a strictly upper triangular matrix.
\end{lemma}
\begin{proof}
    The by-now-familiar strategy is to observe that commuting with $I_2\otimes \AAAA_{n-1}^\heartsuit$ shows that $T\in(\AAAA_n^\heartsuit)'$ has the form 
    
    \[\begin{bmatrix} a_1I+N_1 & B_1 & a_2I+N_2 &B_2\\ 0& a_1I+N_1&0&a_2I+N_2\\
    a_3I+N_3 & B_3&a_4I+N_4&B_4\\0&a_3+N_3&0&a_4I+N_4\end{bmatrix}\]
    where $a_i$ are non-zero constants, the $N_i$ are upper triangular and the $I$ are identity matrices of size $2^{n-2}$.  Then commuting with \[R_\heartsuit\otimes I_{2^{n-2}}=
    \begin{bmatrix} I & 0 & 0 &I\\ 0& 0&I&0\\
    0 & I&0&0\\0&0&0&I\end{bmatrix}\] implies that $a_1=a_4$, $a_3=0$,  and $N_3=B_3=0$, and $N_1=N_4$ as required.
\end{proof}

From this we derive:
\begin{corollary}
    The algebra $\AAAA_n^\heartsuit$ is indecomposable.
\end{corollary}
\begin{proof}
   Let $E$ be any central idempotent in $\AAAA_n^\heartsuit$.  Then by Lemma \ref{heartlem} $E$ is of the form $aI+N$ where $N$ is nilpotent, hence $a\in\{0,1\}$ so $E$ is trivial.
\end{proof}